\RequirePackage{fix-cm}
\documentclass[smallextended]{svjour3}       
\smartqed  
\usepackage{graphicx}
\usepackage{color}
\usepackage[utf8]{inputenc}
\usepackage[english]{babel}
\usepackage{amsmath}
\usepackage{amssymb}
\usepackage{tikz,pgfplots}
\usepackage{diagbox}

\usepackage[square,numbers]{natbib}
\begin{document}

\title{Galerkin Finite Element Method for Nonlinear  Fractional {Differential} Equations}


\author{Khadijeh Nedaiasl         \and
        Raziyeh Dehbozorgi 
}


\institute{K. Nedaiasl \at
             Institute for Advanced Studies in Basic Sciences,
            Zanjan, Iran, \\
            \email{ \texttt{nedaiasl@iasbs.ac.ir  \&   knedaiasl85@gmail.com}}          
           \and
           R.  Dehbozorgi \at
          School of Mathematics, Iran University of Science and Technology, Tehran, Iran,\\
          \email{ \texttt{r.dehbozorgi2012@gmail.com}}
}

\date{Received: date / Accepted: date}

\maketitle

\begin{abstract}
In this paper,  we study  the existence, regularity, and   approximation of the solution for a class of nonlinear fractional differential equations. {In order to do this}, suitable variational formulations are defined for a nonlinear  boundary value problems with Riemann-Liouville and Caputo fractional derivatives together with  the homogeneous Dirichlet condition.   We {investigate} the well-posedness  and also the regularity of the corresponding  weak solutions. Then, we develop a Galerkin finite element approach {for} the numerical approximation of the weak formulations and   { drive a priori error estimates and prove the stability of the schemes}. Finally, some numerical experiments are provided to {demonstrate} the accuracy of the proposed method. 
\keywords{fractional differential operators\and Caputo derivative, Riemann-Liouville derivative\and variational formulation\and nonlinear operators\and   Galerkin method.}
\subclass{ 26A33\and 65R20\and  65J15\and  65L60.}
\end{abstract}

\section{Introduction}\label{intro}
In the current study, we consider {a} fractional order nonlinear boundary value problem as follows:
Find $u$ such that
\begin{equation}\label{asli}
-_{0}^{\circ}D^{s}_x u(x) +g(x,u(x)) =f(x),\quad x\in \Omega:=[0,1],
\end{equation}
\[u(0)=0,\quad u(1)=0,\]
where  $s \in (1,2)$,  and $_{0}^{\circ}D^{s}_x $ refers to either Riemann-Liouville or Caputo fractional derivatives which are detailed below. Furthermore,  $f$ and $g$ are known functions chosen from the suitable function spaces defined in Section \ref{pre}. 

Developing the order of differentiation to any real number is an interesting  question. To find an affirmative answer, some efforts have been done and different types of so-called fractional derivatives have been introduced \cite{kilbas2006theory}. It is notified  that the fractional derivative is a  concept with  attractive  applications in science and engineering. It {appears} in the anisotropic diffusion modeling anomalously for cardiac tissue in microscopic and macroscopic levels. Furthermore, fractional derivative models are certain instances of nonlocal models which are introduced in comparison with the classical ones \cite{dui}. 

Similar to the ordinary and partial differential equations, {one follows} two approaches in seeking solutions for Fractional Differential Equations (FDEs); analytic and numeric solutions. The analytical methods such as {the Fourier, Laplace and Mellin transform methods} and even Green function approach (fundamental solution) are available for some special types of FDEs \cite{kilbas2006theory, podlubny1998fractional}.  In practice, we have to { apply} numerical methods due to the lack of applicability of analytical methods for a wide range of FDEs. Hence, the study of the numerical approaches for them is of great  importance. 

It is worthy to mention that in spite of different definitions for fractional differential operators, most of them are defined by the Abel's integral operator. Among the popular numerical approaches { to} Abel's integral equation, one could mention { the} collocation method \cite{brunner2017volterra} and { the} Galerkin method based on piecewise polynomials \cite{eggermont88, Vogeli2018} where the different varieties of those approaches, in general perspective,  spectral and projection methods, could be utilized to find approximation for FDEs \cite{li2015numerical, podlubny1998fractional}. Converting { FDEs} to a suitable integral equation, and then solving { it} numerically with the above mentioned approaches \cite{wang2018spectral} or directly solving { it} by finite difference method \cite{LI2016614}  are based on the adequate regularity assumptions of the strong solution which is not available in general \cite{ervin2018, jin2015}. {   Here, we contribute { to} the study on the numerical solution of one dimensional nonlinear FDEs which involve Riemann-Liouville or Caputo derivative by introducing  { an appropriate weak formulation} and describing the Galerkin solution with convergence analysis in some appropriate functional spaces.}

The fractional operator in (\ref{asli}) is non-local and $g$ is  a nonlinear function  with respect to $u$, so the study of the existence, uniqueness, regularity of solution,  and the  numerical investigation are challenging.	The existence of classical solution for the nonlinear FDEs is considered in \cite{ZHANG2000804} and for { the linear operator in one dimension}, the regularity of the solution is investigated in \cite{ervin2018}. In this work, we explore the { issues} of the existence and uniqueness with the aid of Browder-Minty method of the monotone operator { theory}. 

In the recent literature, due to their applications in science and engineering \cite{li2015numerical}, several types of numerical methods have been proposed for the approximation of FDEs.  The theory and { the} numerical solution of a linear Riemann-Liouville and Caputo FDEs with two-point boundary condition  have been extensively studied  in \cite{Kopteva2017, Liang2018}. In those works, the fractional boundary value problem is reformulated appropriately in terms of { the} Volterra integral equation, and then  the numerical approach is proceed by some suitable schemes such as the piecewise polynomial collocation and { the} spectral Galerkin methods. Spectral and pseudo-spectral  methods are some of the interesting numerical approaches which are taken into consideration for the FDEs. Among the whole research on this area, we can mention \cite{Li2017, Zaky2019, Yang2015, yarmohammadi} wherein the Jacobi polynomials play a crucial role in the construction of the approximation. The attention to this class of orthogonal polynomials is a motivation for introducing a generalization of them with application in the numerical solution of FDEs \cite{chen2016}. 

In this work, we study the Galerkin finite element method for Riemann-Liouville and Caputo nonlinear fractional boundary value problems of Dirichlet type.  The finite element method is a popular numerical approach in order to find an approximation for  nonlinear differential equations  \cite{HLAVACEK1994168} and  \cite{wriggers2008nonlinear}.  The finite element solution of quasi-linear  elliptic problems with non-monotone  operators have been considered in \cite{abdulle2012priori} and \cite{feistauer1993analysis}. Furthermore, in \cite{Feistauer1986} under the assumptions of the strong monotonicity and Lipschitz continuity of the corresponding second order  nonlinear elliptic  operator, a linear order approximation by finite element method has been obtained. In this paper, the investigated fractional nonlinear  operators have the monotonicity and Lipschitz continuity properties, which are crucial for our analysis. 

The reasonable energy space associated with the non-local operators  is fractional Sobolev space. Also, due to the presence of the nonlinear term in Eq.  (\ref{asli}), we utilize Musielak-Orlicz space in order to  introduce a suitable functional space by intersection of the two mentioned spaces in a convenient way.  Then, with the aid of the monotone operator theory, the coercivity of the nonlinear variational formulation along with the Riemann-Liouville and Caputo fractional derivatives is investigated. This approach leads to { getting} a unique weak solution which { can} be approximated by { the} finite element method.
{ 
	Two main features of this research are as follow:
	\begin{itemize}
		\item[$\bullet$]
		We study the existence and uniqueness issue of the weak solution for the main problem (\ref{asli}) utilizing the monotone operator theory in some  Musielak-Orlicz and fractional Sobolev spaces. 
		\item[$\bullet$]
		We develop the Galerkin finite element method for nonlinear FDEs and obtain a priori error estimates for the method based on the generalization of the C\'{e}a's lemma. 
	\end{itemize}}

We organize the reminder of the paper as follows: in Section \ref{pre}, some introduction regarding to the fractional calculus, { the} semi-linear monotone operator and also { the} suitable { functional} spaces { is} briefly presented.   Section \ref{var} is devoted to the variational formulation of nonlinear boundary value problems along with Riemann-Liouville and Caputo derivatives. Furthermore, the regularity of the solution is studied in this section. The numerical approximation of the weak solution is examined by finite element approximation in Section \ref{fem} along with the full study { of} the existence and uniqueness issue of the discrete equations and the { convergence and stability} of the method.  In numerical experiments section, some FDEs are  solved by { the} finite element method. { Finally}, we provide some conclusion{ s} and further remarks for the future works. 

\section{Preliminaries}\label{pre}
This section is devoted to  some preliminaries to fractional calculus { to provide an} introduction { to} the  { problem considered} in the paper.  The energy space regarding fractional operators, {fractional Sobolev spaces are introduced in this section. Then, in order to { ensure} the existence of the weak solution by monotonicity arguments, some preface to nonlinear functions on { Musielak-Orlicz} spaces { is} provided. 

\subsection{Fractional calculus}
	{ To make this paper self-contained}, we recall the  Riemann-Liouville and { the} Caputo fractional integral and derivatives from \cite{kilbas2006theory}. For any $s >0$ with $n-1<s<n$, $n \in \mathbb{N}$, the right and left sided fractional integrals on the bounded interval $[a,b]$ are as follows:
	
	\noindent left fractional integral operator is defined as
	\begin{equation}\label{eq1}
	(_{a}I^{s}_{x}u)(x)= \dfrac{1}{\Gamma(s)}\int_a^x(x-y)^{s-1} u(y)\mathrm{d}y,
	\end{equation} 
	while the right fractional integral operator is given by
	\begin{equation}\label{eq2}
	(_{x}I^{s}_{b}u)(x)= \dfrac{1}{\Gamma(s)}\int_x^b(y-x)^{s-1} u(y)\mathrm{d}y.
	\end{equation}
	Left-sided Riemann-Liouville fractional derivative of the order $ s $ for { the} function $u \in H^n(\Omega)$ can be defined as
	\begin{equation}
	^{R}_{a}D_x^s u=D^{n}\, _{0}I_{x}^{n -s}u,
	\end{equation}
	where the operator $D^{n}$ denotes the classical derivative of the order $n$. The corresponding right-sided Riemann-Liouville fractional derivative is stated as  
	\begin{equation}
	^{R}_{x}D_b^s u=(-1)^n D^{n} \,_{x}I_{b}^{n-s}u.
	\end{equation}
	In addition, the left-sided Caputo derivative of the order $s$ is given by 
	\begin{equation}
	^{C}_{0}D_x^s u= \,_{0}I_{x}^{s}\, D^n u,
	\end{equation}
	where the following relation defines the right-sided Caputo derivative
	\begin{equation}
	^{C}_{x}D_b^s u={(-1)^n}_{x}I_{b}^{s}D^n u.
	\end{equation}
	From the above definitions, it is apparent that  the Abel's integral operator has a significant role in { defining} fractional derivatives.
	\subsection{Some properties { of} semi-linear operators}
	{ Let $V$ be a real Banach space and $V^{*}$ as its dual space}. {  Let define the duality pairing as the functional 
		\[
		V^*\times V \ni (y,x) \rightarrow \langle y,x \rangle =y(x),
		\]
		which means that $\langle y,x \rangle$ is the value of a continuous linear functional $y \in V^*$ on an element $x \in V$ and $\Vert . \Vert$  and $\Vert . \Vert_{*}$ are the norms associated with $V$ and $V^*$, respectively.} 
	{  Due to the central role of the monotonicity property, we present a formal definition for this concept.} 
	
	\begin{definition}\label{def2}(\cite{askhabov2011nonlinear, krasno}) Let $ V$ be a separable Banach space.
		An operator $ F$ is called monotone on $ V $ if 
		\[\langle Fx-Fy,x-y\rangle \geq 0, \quad \forall x,y\in V,\]
		it is called strictly monotone if
		\[\langle Fx-Fy,x-y\rangle > 0, \quad \forall x,y\in V, \quad x\neq y,\] 
		it is called coercive 
		\[\langle Fx,x \rangle > \gamma (\Vert x \Vert). \Vert x \Vert, \quad \text{where} \quad \gamma(s)\rightarrow \infty  \quad \text{as}\quad s\rightarrow \infty, \] 	
		it is called hemi-continuous if the real-valued function
		\[ s\rightarrow \langle  F(u+s.v), w\rangle, \]	
		is continuous on $[0,1]$ { for  every fixed $u,v$ and $w \in V$. Furthermore,} in the terminology of the article \cite{kato}, the operator $ F:V\rightarrow V^* $  with the domain {  $ D=D(F)\subset V$} is hemi-continuous if for $ u\in D$ { and} $w\in V $, we get $ F(u+t_nw)\rightarrow F(u), $ when the sequence $ t_n $ tends to zero. 
	\end{definition}
	In the following, we state the Browder-Minty theorem which is utilized to prove the existence and uniqueness of the weak solution. 
	\begin{theorem}\label{thmm}
		\textbf{(Browder-Minty)}
		Let $ V $ be a real reflexive Banach space and a hemi-continuous monotone operator 
		$ F:V\rightarrow V^* $ be coercive. Then, for any $ g\in V^* $, there exists a solution $ u^* \in V $
		of the equation 
		\[F(u)=g. \]
		This solution is unique if $F$ is a strictly monotone operator. 
	\end{theorem}
	\begin{proof}
		See the details of the proof in { either} \cite{askhabov2011nonlinear} or { \cite[Chapter 19]{krasno}.}
	\end{proof}
	\subsection{Functional spaces}
	To formulate an appropriate { functional} space
	so that { the problem \eqref{asli}} {  is} well-posed,  we first recall the definition of fractional-order Sobolev spaces. As usual, the standard Lebesgue spaces are
	denoted by $L^{p}\left(  \Omega\right)  $, and their norms by $\left\Vert \cdot\right\Vert _{L^{p}\left(  \Omega\right)  }$. For $p=2$, the scalar
	product is denoted by $\left(  u,v\right)  =\int_{\Omega}u(x){v(x)}\mathrm{d}x$, and the norm by $\left\Vert \cdot\right\Vert =\left(  \cdot,\cdot\right)  ^{1/2}$.
	Let $\lbrace \lambda_n\rbrace_{n\in \mathbb{N}} $ { be} the set of all eigenvalue{ s} of the { following} boundary-value problem
	\begin{align}\label{odesl}
	\begin{split}
	D^2u(x)&=-\lambda u(x), \quad x\in \Omega ,\\
	\frac{du}{dt}(0)&=u(1)=0 ,
	\end{split}
	\end{align}
	where  $ \phi_n $ is an eigenfunction related to $ \lambda_n$ for $ n \in \mathbb{N} $.
	Now, for $ s \in \mathbb{R},$ a Hilbert scale $ H^s(\Omega)$ is defined based on $ \lbrace\phi_n\rbrace_{n\in \mathbb{N}} $ with the  following scalar products and norms
	\begin{equation}\label{iner}
	(u,v)_{H^s(\Omega)}=\sum_{n=1}^{\infty}\lambda^s (u,\phi_n)(v,\phi_n),\quad 
	u,v\in \text{span} \{\phi_n\}_{n\in \mathbb{N}}, 
	\end{equation}
	and
	\[\Vert u\Vert_{H^s(\Omega)}=(u,u)^{\frac{1}{2}}_{H^s(\Omega)}.\]
	Let $\lambda_n:=\mu_n^2$, then  $ \lbrace\mu_n^{-s}\phi_n\rbrace_{n\in \mathbb{N}} $ form{ s} an orthonormal basis for $ H^s(\Omega). $ 	
	It is well-known that $\{ \phi_n \}_{n \in \mathbb{N}}$ is an orthonormal basis for $H^0(\Omega) =L^2{(\Omega)}$, so for $u_{n}=( u, \phi_n)$, we have $u=\sum_{n=1}^{\ \infty} u_{n} \phi_{n}$.
	For any $s\geq 0$, the fractional order Sobolev space is defined  by the spectral properties of the operator \eqref{odesl} and { the} inner product \eqref{iner} as follows 
	\begin{equation}
	H^{s}(\Omega):=\big\{ u \in L^{2}(\Omega) \mid \sum_{k=1}^{\ \infty} \lambda^{s}_{k} u^{2}_k < \infty  \big \},
	\end{equation}
	{ for more details see \cite{antil2017note}. }
	Another approach to define the fractional Sobolev space is using the definition of $L^p(\Omega)$ spaces along with the Slobodeckij  semi-norm \cite{di2012hitchhiker}. { To our aim}, it suffices to set $p=2$ and  let $\left\lfloor s\right\rfloor $ denote the
	largest integer for which $\left\lfloor s\right\rfloor \leqslant s$,  and
	define $\lambda\in\left[  0,1\right[  $,  by $s=\left\lfloor s
	\right\rfloor +\lambda$. For $s\in\mathbb{R}_{>0}\backslash\mathbb{N}$, we
	introduce the scalar product
	
	\begin{align}
	\left(  \varphi,\psi\right)  _{H^{s}(\Omega)} &  :=\sum_{\alpha
		\leqslant\left\lfloor s\right\rfloor }{\left(  D^{\alpha}\varphi,D^{\alpha
		}\psi\right)  }\label{lscalar}\\
	&  +{\int_{\Omega}\int_{\Omega}{\frac{\left(  D^{\left\lfloor s\right\rfloor}\varphi(x)-D^{\left\lfloor s\right\rfloor}\varphi(y)\right)
				{\left(  D^{\left\lfloor s\right\rfloor}\psi(x)-D^{\left\lfloor s\right\rfloor}\psi(y)\right)  }}%
			{|x-y|^{1+2\lambda}}\mathrm{d}x\mathrm{d}y}},\nonumber
	\end{align}
	and the norm $\Vert\varphi\Vert_{H^{s}(\Omega)}:=\left(  \varphi
	,\varphi\right)  _{H^{s}(\Omega)}^{1/2}$. For $s\in\mathbb{N}$, obviously the
	second term in \eqref{lscalar} is ignored. Then, the Sobolev space $H^{s
	}\left(  \Omega\right)  $ is given by
	\[
	H^{s}\left(  \Omega\right)  :=\left\{  u\in L^{2}(\Omega)\mid\forall~0\leq
	k\leq\left\lfloor s\right\rfloor \quad u^{\left(  k\right)  }\in
	L^{2}\left(  \Omega\right)  \text{\quad and\quad}\Vert u\Vert_{H^{s}%
		(\Omega)}<\infty\right\}  .
	\]
	The dual space of $H^{s}(\Omega)$ is denoted by { $H^*(\Omega): = H^{-s}(\Omega)$} and is
	equipped with the {norm}
	\begin{equation}
	\Vert u\Vert_{H^{-s}(\Omega)}:=\sup\limits_{v\in H^{s}(\Omega)}%
	\frac{(u,v)}{\Vert v\Vert_{H^{s}(\Omega)}},
	\end{equation}
	where $(\cdot,\cdot)$ denotes the continuous extension of the $L^{2}$-scalar
	product to the duality pairing $\langle\cdot,{\cdot}\rangle$ in
	$H^{-s}(\Omega)\times H^{s}(\Omega)$.
	Let $ \tilde{H}^{s}(\Omega) $ be the set of { functions in $ H^{s}(\Omega) $ which are extended by zero to the whole domain $ \mathbb{R}$. This space could also be defined by the intermediate space of the order $ s \in (0,1)$  given as}
	\begin{eqnarray}
	H_{0}^s(\Omega)=[H_0^m(\Omega), H^{0}(\Omega)]_{\theta}, \quad m\in \mathbb{Z},
	\end{eqnarray}
	where $ m(1-\theta)=s $  {  and  $H^m_0(\Omega)$ denotes the closure of  $ \mathcal{D}(\Omega)$ in $H^m (\Omega)$  \cite{lions2012non}. In the latter,  $ \mathcal{D}(\Omega)$ stands for the set of all infinitely differentiable functions with compact support which is equipped with the locally convex topology. }
	Indeed, for $ \phi\in {H}_{0}^{s}(\Omega) $, $ \phi $ and its derivatives of { the} order $ k\leq m $ have the compact support property. For $m=1$,  we set 
	$\tilde{H}^{s}(\Omega):= H_{0}^{1-s}(\Omega)$.
	{ Let $I_L$ (respectively $I_R$) denote  half interval $(-\infty, b)$ (res. $(a, \infty)$), then $\tilde{H}^{s}_{L}(\Omega)$ (res. $\tilde{H}^{s}_{R}(\Omega)$) stands for the set of all $u \in {H}^s(\Omega)$ whose extension by zero denoted by $\tilde{ u}$ are in   ${H}^{s}(I_L)$ (res. ${H}^{s}(I_R)$) \cite{adams, jin2015}. }
	
	\begin{theorem}\label{thm1} (\cite{jin2015})
		Assume that  $ n-1<s<n$, for $n\in \mathbb{N}$. The operators $ ^{R}_{0}D_x^s u $ and $ ^{R}_{x}D_1^s u $ for $ u\in \mathcal{D}(\Omega)$ { can} be extended continuously to operators with the same notations from $ \tilde{H}^s(\Omega) $ to $ L^2(\Omega)$, i.e.,
		\begin{equation}
		\Vert ^{R}_{0}D_x^s u \Vert_{L^2(\mathbb{R})}\leq c\Vert u\Vert_{\tilde{H}^s(\Omega)},
		\end{equation}    
		and
		\begin{equation}
		\Vert ^{R}_{x}D_1^s u \Vert_{L^2(\mathbb{R})}\leq c\Vert u\Vert_{\tilde{H}^s(\Omega)}.
		\end{equation}
	\end{theorem}
	The following theorem represents some profitable characteristics of { the} fractional differential and integral operators.
	\begin{theorem}\label{taf}
		The following statements hold:
		\begin{itemize}
			\item[a)] The integral operators $_{0}I_x^s$ and $_{x}I_1^s$ satisfy the semi-group property.
			\item[b)] For $\phi,\psi\in L^2(\Omega)$, $(_{0}I_x^s\phi, \psi)=(\phi,\, _{x}I_1^s\psi)$.
			\item[c)] For any $ s>0, $  the function $ x^{s}\in {H}^\alpha(\Omega),$ where $ 0\leq\alpha <s+\frac{1}{2}. $
			\item[d)]  For any non-negative  $ \alpha, \gamma,$ the Riemann-Liouville integral operator $ I^\alpha $ is a bounded map from $ \tilde{H}^{\gamma}(\Omega) $ into $ \tilde{H}^{\gamma+\alpha}(\Omega). $ 
			\item[e)] The operators $ _{0}^{R}D_{x}^{s}:\tilde{H}^s_L(\Omega)\rightarrow L^2(\Omega)$ and  $ _{x}^{R}D_{1}^{s}:\tilde{H}^s_R(\Omega)\rightarrow L^2(\Omega)$ are continuous.
			\item[f)] For any $ s\in (0,1) $ and $ u\in \tilde{H}^1_R(\Omega)  $, $ _{0}^{R}D_{x}^{s}u=_{0}^{R}I_{x}^{1-s} u^\prime$. Meanwhile, $ u\in \tilde{H}^1_L(\Omega)  $, then $ _{x}^{R}D_{1}^{s}u=- _{x}^{R}I_{1}^{1-s} u^\prime$. 
		\end{itemize}
	\end{theorem}
	\begin{proof}
		{  The first item has been investigated} in \cite[Theorem 2.4]{kilbas2006theory}.	{ The Fubini's Theorem  proves item b \cite[Lemma 2.7]{kilbas2006theory}}. The proof of other items can be found in \cite{jin2015}.
	\end{proof}
\subsubsection{Nonlinear functions on Orlicz spaces}
	Throughout this paper, we require some important properties for the nonlinear  part of Eq. (\ref{asli}). A suitable { functional} space to deal with the monotone operators with nonlinear terms { is the} Orlicz spaces  or { the} generalized Orlicz space {   which is called as Musielak-Orlicz} space  \cite{ adams, bardaro2008nonlinear}.
	We recall some necessary definitions and properties related to the mentioned spaces.  
	{	\begin{definition}
			A function $A: \mathbb{R}\rightarrow \left[0,\infty\right]$	is termed an $N$-function if 
			\begin{itemize}
				\item[a)]  It is even and convex;
				\item[b)]  $A(t)=0$ if and only if $t=0$;
				\item[c)] $\lim_{t\rightarrow 0} \frac{A(t)}{t} = 0$, $\lim_{t\rightarrow \infty} \frac{A(t)}{t} =\infty$.
			\end{itemize}
		\end{definition}
		For more details on $N$-function, one can see \cite[Chapter VIII]{adams}, \cite[Chapter I]{rao} and \cite[Appendix]{warma}. 
		\begin{definition}(\cite[Chapter II]{mendez2019analysis})
			Let $(\Omega, \Sigma, \mu)$ be a  measure space such that $\mu$ is $\sigma$-finite and complete. We say that a function  $\varphi : \Omega \times \mathbb{R} \rightarrow \left[ 0, \infty\right]$  is a Musielak-Orlicz function on $\Omega$ if 
			\begin{itemize}
				\item[a)] $\varphi(.,t)$ is measurable for all $t\in  \left[ 0, \infty\right]$;
				\item[b)] For $\text{ a.e.}~ x\in \Omega, \varphi(x,.)$ is non-trivial, even and convex;
				\item[c)] $\varphi(x,.)$ is vanishing and continuous at $0$ for $\text{ a.e.}~ x\in \Omega$;
				\item[d)] $\varphi(x,t)>0$, $\forall t>0$ and $\lim_{t \rightarrow \infty} \varphi(x,t) = \infty$. 
			\end{itemize}
		\end{definition}
		In this paper, $\Sigma$ is the $\sigma$-algebra of subsets of  $\Omega$ and $\mu$ denotes the Lebesgue measure on $\Sigma$. 
		The complementary Musielak-Orlicz function $\tilde{\varphi}$ is defined by 
		\[
		\tilde{\varphi}(x,t) : =  \sup\big\{ s\vert t\vert - \varphi(x,s)~ |~ s>0 \big\},
		\]
		which is the same as the complementary of the Young function (a more general class of $N$-functions) defined in \cite{warma, rao}.  
	}
	
\noindent	\textbf{Assumption A.}{ (\cite{antil2017note})}
		Assume that a nonlinear function $ g(x,t):\Omega \times \mathbb{R}  \rightarrow \mathbb{R} $  satisfies the following properties
		\[
		\left\{
		\!
		\begin{aligned}
		&g(x,.) \: \text{continuous, odd, strictly monotone,} \quad & \text{a.e. on } \Omega,\\
		&g(x,0)=0, \quad  \lim\limits_{t\rightarrow \infty} g(x,t)=\infty, & \text{a.e. on } \Omega,\\
		&g(.,t) \:\text{is measurable}, & \forall t \in \mathbb{R}.
		\end{aligned}
		\right.
		\]
	
	{ Note that the inverse of the function $ g(x,.) $ exists due to  the strictly monotone property.} Let us denote it by {$\tilde{g}(x,.).$}  We define $G(x,t)$ and $ \tilde{G}(x,t)$ by 
	\[G(x,t):=\int_{0}^{\vert t\vert }g(x,s)\mathrm{d}s,\quad \quad \tilde{G}(x,t):=\int_{0}^{\vert t\vert }\tilde{g}(x,s)\mathrm{d}s.\]
	These functions are complementary Musielak-Orlicz functions that are ${N} $-functions { with} respect to the second variable.
	\begin{definition}\label{def1}
		Let $ G(x,.) $ be an ${N}$-function. We say this function satisfies the global ($ \Delta_2 $)-condition if there exists a constant $ c\in(0,1] $ such that for $ \text{ a.e.} ~x\in \Omega  $ and for all $ t\geq 0 $
		\[c\, t\, g(x,t)\leq G(x,t)\leq t\,g(x,t),\]
		where the function $ g(x,t) $ { satisfies in } the Assumption A.
	\end{definition}
	{Let $\mathcal{M}(\Omega)$ represents all real-valued measurable functions defined on $\Omega$, then  
		the Musielak-Orlicz space generated by $G$ is defined as follows
		\[L_G(\Omega):=\Big\lbrace u \in \mathcal{M}(\Omega) \mid G(.,u(.))\in L^1(\Omega)\Big\rbrace,\]
		which means that the modular
		\begin{equation*}
		\begin{split}
		\rho_G&: \mathcal{M}(\Omega) \rightarrow \left[0,  \infty\right],\\
		\rho_{G} (u)&:= \int_{\Omega}G(t,u(t))\mathrm{d}t,
		\end{split}
		\end{equation*}
		is measurable \cite{bardaro2008nonlinear}.
		If $G(x,.)$ and $\tilde{G}(x,.)$ satisfy  Definition \ref{def1}, then by the Theorem 8.9 in \cite{adams}, it is concluded that $L_G(\Omega)$   equipped with the Luxemburg norm;
		\[\Vert u\Vert_{G,\Omega}:=\inf\Big\lbrace m>0 \mid \rho_{G}(\frac{u}{m})\leq 1\Big\rbrace,\]
		{ is a reflexive Banach space}.
		The same result is { valid} for the space $ L_{\tilde{G}}(\Omega)$ with the norm $\Vert .\Vert_{\tilde{G},\Omega}$.}
	Moreover, in our analysis we need  the generalized H\"{o}lder inequality  given by
	\begin{equation}
	\Big\vert \int_\Omega u(t)v(t)  \mathrm{d}t \Big\vert \leq 2\Vert u\Vert_{G,\Omega} \Vert v \Vert_{\tilde{G},\Omega}, \quad \forall u \in L_{G}(\Omega), ~~ \forall v\in L_{\tilde{G}}(\Omega).
	\end{equation} 
	{ There are some proofs for this relation that one can see them in  \cite[Chapter 3]{function} and \cite{adams}.} 
	Furthermore, the following important result
	\begin{equation}\label{inf}
	\lim\limits_{\Vert u\Vert_{G, \Omega}\rightarrow \infty} \dfrac{\rho_{G}(u)}{\Vert u\Vert_{G, \Omega}}=\infty,
	\end{equation}
	which is obtained in \cite{warma} has a significant role in the applicability of the monotone operator theorems for our target. 
	\begin{lemma}\label{lem1}(\cite{warma})
		If $ g(.,u(.)) $ satisfies { the} ($ \Delta_2 $)-condition, then for all  $u \in L_{G}(\Omega)$ one can get $ g(.,u(.))\in L_{\tilde{G}}(\Omega)$.
	\end{lemma}
	Now, we are ready to define the suitable function space which is appropriate for our problem.
	\begin{definition}\label{defv}
		Let $1< s<2$. Under the Assumption A { which has been} fulfilled in Definition \ref{def1}, consider the following  reflective Banach space $ U $ as
		\[U:=U(\Omega, G):=\lbrace \phi \in \tilde{H}^\frac{s}{2}(\Omega) \mid G(.,\phi)\in L^1(\Omega) \rbrace,\]
		where equipped with the norm
		\begin{equation}\label{normv}
		\Vert u\Vert_U:=\Vert u\Vert_{\tilde{H}^\frac{s}{2}(\Omega)}+\Vert u\Vert_{{ G,\Omega}}.
		\end{equation}
	\end{definition}
	We state the following lemma from \cite{antil2017note} which {  is important } in the investigation of { the} regularity of the solution.
	\begin{lemma}\label{lem0}
		Let $0 \leq  s < 1$ and assume that for $M>0$, there exists a constant $l_{M}$ such that 
		$g$ satisfies 
		\begin{equation}\label{lip}
		\vert g(x,u_{1}) -g(x,u_{2})  \vert\leq l_{M} \vert  u_{1}- u_{2} \vert, \quad x,y \in  \Omega,\quad  u_{i}\in \mathbb{R} \ \text{with} \  \vert u_{i} \vert \leq M. 
		\end{equation}
		Then, for $u \in H^{s}(\Omega)\cap L^{\infty}(\Omega) $, we have $g(.,u(.)) \in H^{s}(\Omega)$. 
	\end{lemma} 
	
	\section{Variational formulation and regularity}\label{var}
	In this section, we aim to work with { an} appropriate variational formulation  to overcome the difficulty of dealing with the nonlinear and fractional terms of the main problem for both case{ s} of { the} Riemann-Liouville and the Caputo derivatives separately. The non-local variational problems  possess reduced order  smoothing properties which { are} investigated in this section. 
	\subsection{The Riemann-Liouville fractional operator}
	The appropriate variational formulation of the  problem \eqref{asli} { in the linear case (with $ g(x,u)=0 $)} introduced in \cite{jin2015}
	is:  
	
	Find {  $ u\in U=\tilde{H}^{\frac{s}{2}}(\Omega) $} such that
	\begin{equation}
	\mathcal{A}(u,v)=(f,v), \quad v\in V=U, 
	\end{equation}
	where
	\begin{equation}
	\mathcal{A}(u,v):=-(_{0}^{R}D^\frac{s}{2}_x u\, ,\,  _{x}^{R}D^\frac{s}{2}_1 v),
	\end{equation}
	and $f\in L^2(\Omega)$.
	
	Considering the above form for the fractional part has some pros { in} utilizing the nice properties indicated in Theorem \ref{taf} for the approximation procedure. 
	Hence{ ,} for  { the nonlinear} problem \eqref{asli}, the weak formulation  is stated as follows:
	
	Find $u \in U$	satisfying
	\begin{equation}\label{weak}
	\begin{split}
	\mathcal{L}(u,v)&:= \mathcal{A}(u,v)+ \mathcal{B}(u,v)=\langle f,v\rangle:=F(v),\quad \quad v \in U,   
	\end{split}
	\end{equation}
	where  $f \in U^* $, $ \mathcal{B}(u,v):=(g(x,u),v) $ and $ U^* $ {   is the dual space of the reflexive Banach space $ U$ introduced in Definition \ref{defv} and their duality map is denoted by $\langle . , . \rangle$}.
	\begin{theorem}\label{thmd}
		Let $ 1<s<2 $ and $ u\in \tilde{H}^{\frac{s}{2}}(\Omega),$ the operator  $ \mathcal{A}(u,v) $ is coercive and monotone, i.e.,
		\[
		\exists ~c>0 \quad \text{s.t}\quad
		\mathcal{A}(u,u)\geq c\Vert u \Vert^2_{\tilde{H}^{\frac{s}{2}}(\Omega)}.\]
	\end{theorem}
	\begin{proof}
		It is easily verified that 
		\[\mathcal{A}(u,u)=-(_{0}^{R}D^s_xu,u).\]	
		Let us { consider} $ Su(x):= -_{0}^{R}D^s_xu$ and borrow the notation  $ S^\varepsilon  u $ for $ \varepsilon>0 $ from \cite{eggermont88} as follows
		\[S^\varepsilon u(x)=\frac{-1}{\Gamma(2-s)}\frac{\mathrm{d^2}}{\mathrm{d}x^2}\int_{0}^{x}(x-t)^{1-s}e^{-\varepsilon(x-t)}u(t)\mathrm{d}t, \quad x>0.\]
		Using the Plancherel theorem, we get
		\begin{equation} \label{sh}
		(S^{\varepsilon} u,u)=\int_{-\infty}^{\infty}(S^\varepsilon u\hat{)}(w) \overline{\hat{u}(w)}\mathrm{d}w,
		\end{equation}
		where the notation $\,\hat{}\,$ refers to the Fourier transform.  
		Let us introduce 
		\[
  	   \hat{a}_\varepsilon(w)= w^2 \int_{0}^{\infty}x^{1-s}
  	   e^{-(\varepsilon+\text{i} w)x}\mathrm{d}x,
		\]
		therefore, 
		$(S^\varepsilon u\hat{)}(w)=\hat{u}(w)\hat{a}_\varepsilon(w)$. Now, regarding the principal value of the power function, we write
		$\hat{a}_\varepsilon(w)=w^2(\varepsilon+\text{i} w)^{s-2}$. 
		Hence, 
		\[
		\text{Re}\hat{a}_\varepsilon(w)>\text{Re} \hat{a}_0 (w)=\cos(\dfrac{\pi(2-s)}{2})\vert w\vert^{s}.
		\]
		From Eq. \eqref{sh} and the above equation, we have
		\begin{equation}\label{shat}
		\begin{split}
		\text{Re}(u,S^\varepsilon u)&=\text{Re}(\int_{-\infty}^{\infty}\vert \hat{u}(w)\vert^2 \hat{a}_\varepsilon(w) \mathrm{d}w)\\
		&\geq\cos(\dfrac{\pi(2-s)}{2})\int_{-\infty}^{\infty}\vert \hat{u}(w)\vert^2 \vert w\vert^{s} \mathrm{d}w. 
		\end{split}
		\end{equation}
		From the contradiction argument investigated in \cite[Lemma 4.2]{jin2015}, we conclude the coerciveness of the operator $ \mathcal{A}. $ 
		According to this result, one can deduce the monotonicity of the operator $ \mathcal{A} $ by the Definition \ref{def2}, i.e.,
		\[\mathcal{A}(u,u-v)-\mathcal{A}(v,u-v)>0. 
		\]
	\end{proof}
	{ 
		\begin{remark}
			It should be noticed that results of Theorem \ref{thmd} are well-known \cite{jin2015}, to make the paper self-contained, we have provided another proof based on the interesting  results  for the Abel's integral operator in \cite{eggermont88}. 
	\end{remark}}
	In { the} next theorems, we assert { that the results} guarantee the existence of { a} unique weak solution for Eq. \eqref{asli} along with the Riemann-Liouville derivative.
	\begin{theorem}\label{thm8}
		Suppose that { the} Assumption A holds for the function $ g(x,u) $ which satisfies the global ($ \Delta_2 $)-condition. Consider the variational form \eqref{weak}; then, for every  $ g\in U^* $ and $ 1<s<2 $, Eq. \eqref{asli} has a unique weak solution. 
	\end{theorem}
	\begin{proof}
		Let  $ u\in U $ be fixed. Regarding  { the Lemma \ref{lem1}}, $g(.,u(.))\in L_{\tilde{G}}(\Omega).$ Now, using H\"{o}lder inequality, we obtain
		\begin{equation}
		\vert \mathcal{L}(u,v)\vert\leq \Vert _{0}^{R}D^\frac{s}{2}_x u\Vert_{L^2(\Omega)}\Vert_{x}^{R}D^\frac{s}{2}_1 v\Vert_{L^2(\Omega)}+2\Vert g(.,u)\Vert_{\tilde{G}} \Vert v\Vert_U.
		\end{equation}
		Applying Theorem \ref{thm1} and Lemma \ref{lem1}, the above inequality { can} be simplified as 
		\begin{equation}
		\vert \mathcal{L}(u,v)\vert\leq \big(\Vert _{0}^{R}D^\frac{s}{2}_x u\Vert_{L^2(\Omega)}+2\Vert g(.,u)\Vert_{\tilde{G}}\big) \Vert v\Vert_U,
		\end{equation}
		which means that the operator $ \mathcal{L} $ is bounded.
		Since $ \mathcal{L}(u,.) $  is linear with respect to the second variable,  $ \mathcal{L}(u,.)\in U^*$  for all $ u\in U $. 
		Due to the  { strict monotonicity property}  of $ g(x,.) $  and  Theorem \ref{thmd}, { one can conclude the monotonicity of the operator $ {\mathcal{L}} $}.
		{ With} regard to the previous theorem, we have the coercivity of the operator $ \mathcal{A} $. Under the assumption about the function $ g(x,.) $ and Eq. (\ref{inf}), we { get that}
		\[\lim\limits_{\Vert u\Vert_{G, \Omega}\rightarrow \infty} \dfrac{\rho_{G}(u)}{\Vert u\Vert_{G, \Omega}}=\infty,\]
		therefore, $ \mathcal{L} $ is coercive. 
		Here, we want to show the hemi-continuity of the nonlinear monotone operator $ \mathcal{L}. $ To this end, let $ u, w\in \mathbb{R}$ and the sequence $ t_n $  tend to zero. The { objective} is to show that $ \mathcal{L}(u+t_nw,v)$ tends to $\mathcal{L}(u,v).$
		It is an evident fact that $ (_{0}^{R}D^\frac{s}{2}_x (u+t_nw), _{x}^{R}D^\frac{s}{2}_1 v) $ converges to $ (_{0}^{R}D^\frac{s}{2}_x u,_{x}^{R}D^\frac{s}{2}_1v) $
		when 	$ t_n$ tends toward zero. Regarding the continuity of the function $ g(x,.)$, the  claim on the operator $ \mathcal{L} $ being hemi-continuous is verified. 	
		Consequently, the existence of { a} unique weak solution is proved by utilizing { the} Browder-Minty Theorem \ref{thmm}. 
	\end{proof} 
	
	We proceed { with} the discussion on the regularity of the solution.  { In order to do this}, the following theorem is stated. 
	\begin{theorem}\label{thmr1}
		Consider Eq. \eqref{asli} along with the Riemann-Liouville fractional derivative where the function $ g(x,u) $  satisfies the Lipschitz condition \eqref{lip}. Then, this equation has a solution which fulfills  the nonlinear Volterra-Fredholm integral equation of the form
		\begin{equation}
		u(x)=x^{s-1}\big(_{0}I^{s}_{x}(g(.,u(.))-f(.))\big)(1)-~_{0}I^{s}_{x}\big(g(.,u(.))-f(.)\big)(x).
		\end{equation} 
		In addition, let $ u\in U$  be a weak solution of Eq. \eqref{asli}. Then, $ u\in \tilde{U} $ where $ \tilde{U}:=\big\lbrace u \in {H}^\alpha(\Omega) \cap \tilde{H}^{\frac{s}{2}}(\Omega) \mid G(.,u(.))\in L^1(\Omega) \big\rbrace $ for $ 0\leq \alpha \leq s-\frac{1}{2}.$
	\end{theorem}
	\begin{proof}
		{According to the argument  about converting the FDEs into { integral equations} in Chapter 5 of \cite{diethelm2010analysis} and by adjusting the homogeneous Dirichlet  boundary conditions, $ u(x) $ takes the following form
			\begin{equation}
			u(x)=w\,x^{s-1}-\dfrac{1}{\Gamma(s)}\int_0^x (x-y)^{s-1}\big(g(y,u(y))-f(y)\big)\mathrm{d}y,
			\end{equation} }
		where  $ w=\big(_{0}I^{s}_{x}(g(.,u(.))-f(.))\big)(1).$ Moreover, by means of Theorem \ref{taf} part ($ c $), we have $ x^{s-1}\in {H}^\alpha(\Omega),$ for $ 0\leq\alpha <s-\frac{1}{2}.$ On the other hand,   Lemma \ref{lem0} and $ u\in U $ insure that $ g(.,u(.))\in {H}^{\frac{s}{2}}(\Omega)$ which is a subset of $ L^2(\Omega)$. Hence, it { achieves} that $ g(.,u(.))-f(.)\in L^2(\Omega) $. In addition,  we conclude from part ($ d $) of Theorem \ref{taf} that  $_{0}I^{s}_{x}(g(.,u(.))-f(.))\in \tilde{H}^{s}(\Omega)$. 
		Consequently,  since $u\in\tilde{H}^{\frac{s}{2}}(\Omega) $,  one can deduce that $ u\in {H}^\alpha(\Omega) \cap \tilde{H}^{\frac{s}{2}}(\Omega) ,\,$ for $  0\leq\alpha <s-\frac{1}{2}. $ 
	\end{proof}
	\begin{remark} Due to the presence of the singular term $ x^{s-1}$, it is apparent that the best possible regularity of the solution \eqref{asli} { can} occur in ${H}^\alpha(\Omega)$. The similar argument about the regularity of the linear form of Eq. \eqref{asli} reported in Theorem 4.4 and Remark 4.5 of the interesting work \cite{jin2015} verifies the above claim. In that work, ${H}^\alpha(\Omega),~ 0\leq\alpha <s-\frac{1}{2}$ is displayed by ${H}^{s-1+\alpha}(\Omega)$, where $1-\frac{s}{2}\leq\alpha <\frac{1}{2}$, to show the presence of the singular term better. 
	\end{remark}
		\subsection{The Caputo fractional operator}
	As  discussed in \cite{jin2015}, the difference between the variational formulation of the Caputo and the Riemann-Liouville equations is their admissible test spaces. It means that the variational formulation of Eq. \eqref{asli} along with the Caputo derivative is:
	
	{Find $u \in U$ such that 
		\begin{equation}\label{weak2}
		\begin{split}
		\mathcal{L}(u,v)&:= \mathcal{A}(u,v)+ \mathcal{B}(u,v)=\langle f,v\rangle,\quad  \,v \in V, 
		\end{split}
		\end{equation}
		where 
		\[U:=\lbrace\phi \in \tilde{H}^{\frac{s}{2}}(\Omega)\mid G(.,\phi(.))\in L^1(\Omega)\rbrace,\]
		and
		\begin{equation}\label{Vv}
		V:=\lbrace\phi \in \tilde{H}^{\frac{s}{2}}(\Omega)\mid (x^{1-s},\phi)=0 \rbrace.
		\end{equation}}
	
	In order to define { the} appropriate test space $ V $ for the Caputo case, we assume that $ \phi^{*}(x)=(1-x)^{1-s}$, which belongs to $ \tilde{H}^{\frac{s}{2}}(\Omega)$, and apparently for any  $\phi \in U$, we have $ \mathcal{A}(\phi,\phi^{*})=0.$ In the Caputo fractional derivative case, we set $V=\text{span} \big\lbrace \tilde{\phi}_i(x)=\phi_i(x)-\gamma_{i}(1-x)^{s-1}~|~ i=0, \dots, N\big\rbrace$\\ where
	\begin{equation}\label{gammaha}
	\gamma_{i}=\dfrac{(x^{1-s},\phi_i(x))}{(x^{1-s},(1-x)^{s-1})},
	\end{equation}
	and $ \phi_i\in U.$ We { will} elucidate the above argument in the next section.
	Note that both Theorems \ref{thmd} and \ref{thm8} are valid for the operators including Caputo derivative which have the same variational formulations for the Riemann-Liouville counterpart. 
	Now, we discuss the regularity of the solution by the following theorem.
	\begin{theorem}\label{thmr2}
		Let us consider Eq. \eqref{asli}  with { the} Caputo fractional derivative for $ 1<s<2 $ in which the function { $ g(x,u) $}  satisfies the Lipschitz condition \eqref{lip} and $ f\in {H}^\alpha(\Omega) $ so that $ \alpha + s \in (\frac{3}{2},2) $ and $ \alpha\in [0,\frac{1}{2}) $. Then, this equation has a solution which fulfills the following nonlinear  Volterra-Fredholm integral equation
		\begin{equation}
		u(x)=\,x \,_{0}I^{s}_{x}\big(g(.,u(.))-f(.)\big)(1)-~ _{0}I^{s}_{x}\big(g(.,u(.))-f(.)\big)(x).
		\end{equation} 
		In addition, let $ u(x)\in U$  be a weak solution of Eq. \eqref{asli}. Then $ u\in \tilde{U} $ where 
		{ $ \tilde{U}:=\Big\{  \phi(x) \in {H}^{\alpha+s}(\Omega) \cap \tilde{H}^{\frac{s}{2}}(\Omega) \mid  G(.,u(.))\in L^1(\Omega) \Big\} $.} 
	\end{theorem}
	\begin{proof}
		According to \cite[Theorem 6. 43]{diethelm2010analysis}, $ u(x) $ has the following form
		\begin{equation}
		u(x)=w\,x-\dfrac{1}{\Gamma(s)}\int_0^x (x-y)^{s-1}\big(g(y,u(y))-f(y)\big)\mathrm{d}y,
		\end{equation} 
		where $ w=\big(_{0}I^{s}_{x}(g(.,u(.))-f(.))\big)(1)$ is determined by adjusting the boundary condition.  
		On the other hand,   Lemma \ref{lem0} and $ u\in U $ imply that $ g(.,u(.))\in {H}^{\frac{s}{2}}(\Omega)$ which is a subset of  ${H}^{\alpha}(\Omega)$. Therefore, $ g(.,u(.))-f(.)\in  {H}^{\alpha}(\Omega) $.  Hence, by part ($ d $) of Theorem \ref{taf}, we have  $_{0}I^{s}_{x}(g(.,u(.))-f(.))\in {H}^{\alpha+s}(\Omega)$. Moreover, by means of Theorem \ref{taf} part ($ c $), we have $ x\in \tilde{H}^\beta(\Omega), $ for $ 0\leq\beta <\frac{3}{2}.$  Consequently, from the inclusion argument and $u\in\tilde{H}^{\frac{s}{2}} (\Omega)$, we can conclude that  $ u\in {H}^{\alpha+s}(\Omega) \cap \tilde{H}^{\frac{s}{2}}(\Omega)$ where $ 0\leq\alpha <\frac{1}{2}. $ 
		
	\end{proof}
	\begin{remark}
		As observed in Theorems \ref{thmr1} and \ref{thmr2}, owing to the existence of  the  intrinsic singular term $ x^{s-1} $ in the solution representation, 
		the solution of {the} differential equation with { the} Riemann-Liouville derivative has less regularity in comparison with the Caputo fractional counterpart. In fact, the best possible regularity in the Riemann-Liouville fractional derivative case belongs to  $\tilde{H}^{s-1+\alpha}(\Omega)$. It is worthy to note that the  superiority for the Caputo fractional derivative case comes from the fact that the function under the Caputo derivative  is supposed to  be twice differentiable.
	\end{remark}
	{  In the following remark, the stability of the variational formulations is shown. 
		
		\begin{remark}\label{stab}
			If we assume that $f \in \tilde{ H}^{-\frac{s}{2}}(\Omega) \hookrightarrow U^*$ and take $u=v$ in the relation \eqref{weak}, then by using the results of Theorem \eqref{thmd} and the relation  $\mathcal{B}(u,u) \geq 0$, we see that there is a constant $c>0$ such that 
			$c \Vert u\Vert^2_{\tilde{ H}^{\frac{s}{2}}(\Omega)} \leq \mathcal{A}(u,u) + \mathcal{B}(u,u)$. 		
			Note that $ \langle f,u \rangle_{U^*, U} =\langle f,u \rangle_{\tilde{ H}^{-\frac{s}{2}}, \tilde{ H}^{\frac{s}{2}}} $, so we get that  
			\[
			c \Vert u\Vert^2_{\tilde{ H}^{\frac{s}{2}}(\Omega)}	\leq  \vert \langle f,u \rangle \vert\leq \Vert f\Vert_{\tilde{ H}^{-\frac{s}{2}}(\Omega)} \Vert u \Vert_{\tilde{ H}^{\frac{s}{2}}(\Omega)},
			\]
			which means that there is a constant $C>0$ such that $\Vert u\Vert^2_{\tilde{ H}^{\frac{s}{2}}(\Omega)} \leq C \Vert f\Vert_{\tilde{ H}^{-\frac{s}{2}}(\Omega)}$. 
			
	\end{remark}}
	
	\section{Finite element approximation}\label{fem}
	{ To find an approximate solution,} we  discretize the continuous problem \eqref{weak} by a Galerkin  finite element method. { In order to do this}, a piecewise polynomial finite element method is introduced  over the interval $ \Omega=[0, 1]$. 
	Let us  define $\mathbb{P}_{r}(\Omega)$ as the space of univariate polynomials of the degree less than or equal to $r$, for positive integer $r$ { and } $\chi_{h}$ be a uniform mesh partition on $\Omega$, given by 
	\begin{equation}\label{mesh}
	0=x_{0}<x_{1}<\dots<x_{N-1}<x_{N}=1, \quad N \in \mathbb{N},
	\end{equation}
	with fixed mesh size $ h_{i}= x_{i}-x_{i-1} $. The set $\chi_{h}$ induces a mesh $\mathcal{T}_{h}=\{\tau_{i} | 1\leq i \leq N\}$ on $\Omega$, where  
	$\tau_{i}= [x_{i-1}, x_{i}]$. The length of a subinterval $\tau\in\mathcal{T}_{h}$ is denoted by $h_{\tau}$ and the maximal mesh width by $h:=\max\left\{  h_{\tau}:\tau\in\mathcal{T}_{h}\right\}  $.
	We choose standard continuous and  piecewise polynomial function space of { the} degree $r\in\mathbb{N}$ on
	$[0,1]$ { defined} by
	\begin{equation}
	S_{\mathcal{T}}^{r}(\Omega):=
	\{v\in C(\Omega):\left.  v\right\vert _{\tau}\in\mathbb{P}_{r}\left(
	\tau\right)  ,\forall\tau\in\mathcal{T}\}.
	\end{equation}
	The nodal points are given by%
	\[
	\mathcal{N}_{r}:=
	\left\{  \xi_{i,j}:=x_{i-1}+j\frac{x_{i}-x_{i-1}}{r}\quad1\leq i\leq N\text{,
	}0\leq j\leq r-1\right\}  \cup\left\{  1\right\}.
	\]We choose the usual standard Lagrange basis functions
	$b_{i,j}^{(r)}$ of $S_{\mathcal{T}}^{r}(\Omega)$. 
	Now, with these piecewise functions, one can define the discrete admissible space  which is a subspace of $S^{r}_{\mathcal{T}}(\Omega)\bigcap H^{1}_{0}(\Omega)$ denoted by $A_{h}$. 
	Particularly, we focus on the linear elements in the numerical experiments. Let $\mathcal{I}_{h}$ be the Lagrange interpolation operator mapping into $A_{h}$.
	We denote the finite element test and trial spaces $U_{h}$ for the Riemann-Liouville fractional derivative with $A_{h}$ which is described above. In order to investigate the Caputo fractional derivative case, we consider  finite dimensional set $U_{h}= A_{h}$ as the trial space.  
	In addition, to construct a suitable  test  space  $V_{h}$, let  $V_{h}=\text{span} \Big\lbrace \tilde{\phi}_i(x) ~|~ i=0, 1,  \dots, N \Big\rbrace$ where
	\[ \tilde{\phi}_i(x)=\phi_i(x)-\gamma_{i}(1-x)^{s-1},\]
	{ and} $\gamma_{i}$ is { given} by \eqref{gammaha}.
	{Finally, the discrete variational formulation released from \eqref{weak} and \eqref{weak2} is:
		
		Find $u_{h} \in A_{h}$ such that 
		\begin{equation}\label{femd}
		\mathcal{L}(u_{h},v_{h})=F(v_{h}), \quad \forall v_{h}\in V_{h}.
		\end{equation}
		We notice that in the approximation procedure, one can use the property ($f$) of Theorem \ref{taf} which means that for the computation of $_0^R D_{x}^{\frac{s}{2}}u_{h}$, one can utilize the relation
		\begin{equation}\label{37}
		\begin{split}
		_0^R D_{x}^{\frac{s}{2}} \phi_{i} = _0^R I_{x}^{1-\frac{s}{2}} \phi_{i}=&\frac{1}{\Gamma(1-\frac{s}{2})}\int_{0}^{x}(x-t)^{-\frac{s}{2}}\phi_{i}'(t)\mathrm{d}t\\
		=&\frac{1}{\Gamma(1-\frac{s}{2})}\int_{0}^{x}(x-t)^{-\frac{s}{2}}(\frac{\chi_{[x_{i-1},x_{i}]}}{h_{i}}-\frac{\chi_{[x_{i},x_{i+1}]}}{h_{i+1}} )\mathrm{d}t\\
		=&\frac{1}{\Gamma(1-\frac{s}{2})}\Big[h^{-1}_{i}\big((x-x_{i-1})_{+}^{1-\frac{s}{2}}  -(x-x_{i})_{+}^{1-\frac{s}{2}} \big)\\
		& - h^{-1}_{i+1}\big((x-x_{i})_{+}^{1-\frac{s}{2}} 
		-(x-x_{i+1})_{+}^{1-\frac{s}{2}} \big)\Big],
		\end{split}
		\end{equation}
		where $a_{+}=\max\{a,0\}$, and analogously for $_x^R D_{1}^{\frac{s}{2}}u $, we apply
		\begin{equation}\label{38}
		\begin{split}
		_x^R D_{1}^{\frac{s}{2}} \phi_{i} =- _x^R I_{1}^{1-\frac{s}{2}} \phi_{i}=&- \frac{1}{\Gamma(1-\frac{s}{2})}\int_{x}^{1}(x-t)^{-\frac{s}{2}}\phi_{i}'(t)\mathrm{d}t\\
		=&\frac{1}{\Gamma(1-\frac{s}{2})}\int_{x}^{1}(x-t)^{-\frac{s}{2}}(\frac{\chi_{[x_{i-1},x_{i}]}}{h_{i}}-\frac{\chi_{[x_{i},x_{i+1}]}}{h_{i+1}} )\mathrm{d}t\\
		=&\frac{1}{\Gamma(1-\frac{s}{2})}\Big[h^{-1}_{i}\big((x_{i}-x)_{+}^{1-\frac{s}{2}}  -(x_{i-1}-x)_{+}^{1-\frac{s}{2}} \big)\\
		& - h^{-1}_{i+1}\big((x_{i+1}-x)_{+}^{1-\frac{s}{2}}
		  -(x_{i}-x)_{+}^{1-\frac{s}{2}} \big)\Big].
		\end{split}
		\end{equation}
		Therefore, the term $\mathcal{A}(\phi_{i}, \phi_{j})= \big( _0^R D_{x}^{\frac{s}{2}}\phi_{i},\,  _x^R D_{1}^{\frac{s}{2}}\phi_{j} \big)$ can be derived by the above arguments for the Riemann-Liouville derivative case. For the Caputo fractional derivative, we have $\mathcal{A}(\phi_{i},\tilde{ \phi_{j}})= \big( _0^R D_{x}^{\frac{s}{2}}\phi_{i},\,  _x^R D_{1}^{\frac{s}{2}}\tilde{ \phi_{j}} \big)$ which can be simplified as 
		\[
		\big( _0^R D_{x}^{\frac{s}{2}}\phi_{i}\, ,\,  _x^R D_{1}^{\frac{s}{2}}\phi_{j} \big) - \gamma_{j} \big( _0^R D_{x}^{\frac{s}{2}}\phi_{i}\, ,\,  _x^R D_{1}^{\frac{s}{2}}(1-x)^{s-1} \big).
		\]
		The first term { can} be computed using  the relations \eqref{37} and \eqref{38} and the second term is disappeared,  because 
		\begin{equation}
		\begin{split}
		\big(  _0^R D_{x}^{\frac{s}{2}}\phi_{i}\, ,\, _x^R D_{1}^{\frac{s}{2}}(1-x)^{s-1}\big)& = - \big(_0 I_{x}^{1-\frac{s}{2}}\phi'_{i}\, , \, _x^R D_{1}^{\frac{s}{2}}(1-x)^{s-1}\big)\\
		&= c_{\alpha} \big(\phi'_{i}\, , \, _0 I_{x}^{1-\frac{s}{2}}(1-x)^{\frac{s}{2}-1}\big)\\
		&= c_{\alpha}\big(\phi'_{i}\, , \, 1\big)\\
		& =0,
		\end{split} 
		\end{equation}
		where $c_{\alpha}$ is a constant depending on $\alpha$. 
		\subsection{Convergence analysis}\label{secconv}
		This section is devoted to the study of the { approximate solution achieved  in the} previous section.  For this end, we consider the existence and uniqueness issue for the discrete equation. In addition, we find an appropriate priori error bound.
		
		\begin{theorem}
			The discrete problem  \eqref{femd} has a unique solution.
		\end{theorem}
		\begin{proof}
			The existence of the discrete solution { can be verified through} the Browder-Minty theorem with the same argument pursued in Section \ref{var}. For the uniqueness  issue, 
			let $u_1$ and $ u_2 $ be finite element solutions of \eqref{weak}. Hence, 
			\begin{equation}\label{as}
			\begin{split}
			0= \mathcal{L}(u_1,v_h)- \mathcal{L}(u_2,v_h)&=\mathcal{A}(u_1,v_h)-\mathcal{A}(u_2,v_h)+\mathcal{B}(u_1,v_h)-\mathcal{B}(u_2,v_h)\\
			=&-\int_{\Omega}\,_{0}^{R}D_{x}^{\frac{s}{2}} (u_1-u_2)(x)\,_{x}^{R}D_{1}^{\frac{s}{2}}v_h(x)\mathrm{d}x\\
			&+\int_{\Omega} \big(g(x,u_1(x))-g(x,u_2(x))\big)v_{h}(x)\mathrm{d}x.
			\end{split}
			\end{equation}
			Since the operator $ \mathcal{A}(u,v) $ is coercive,  for $ v=u_1-u_2 $, we get that
			\[- \int_{\Omega}\, _{0}^{R}D_{x}^{\frac{s}{2}} (u_1-u_2)(x) \, _{x}^{R}D_{1}^{\frac{s}{2}}(u_1-u_2)(x)\mathrm{d}x\geq c\Vert u_1-u_2 \Vert^2_{\tilde{H}^{\frac{s}{2}}(\Omega)}.
			\]
			Using { the} above equation and  Eq. \eqref{as}, one can conclude that
			\begin{align*}
			c\Vert u_1-u_2 \Vert^2_{\tilde{H}^{\frac{s}{2}}(\Omega)}&+\int_{\Omega} \big(g(x,u_1(x))-g(x,u_2(x))\big)(u_1(x)-u_2(x))\mathrm{d}x\\
			&\leq - \int_{\Omega}\, _{0}^{R}D_{x}^{\frac{s}{2}} (u_1-u_2)(x) \, _{x}^{R}D_{1}^{\frac{s}{2}}(u_1-u_2)(x)\mathrm{d}x\\
			&+\int_{\Omega} (g(x,u_1(x))-g(x,u_2(x)))(u_1(x)-u_2(x))\mathrm{d}x\\
			&= 0.
			\end{align*}
			Since $ g(x,u) $ is a monotone function with respect to the second variable, thereby the above inequality, we conclude the uniqueness of the approximate solution.
		\end{proof}
		\begin{lemma}\label{lemcars}
			Let $ \mathcal{T}_h $ be a uniform mesh on $ \Omega. $ For real numbers $ s, \, m $ with $ m\geq \frac{s}{2} $, and also $S_\mathcal{T}^r (\Omega)$ with an integer $ r\geq 0 $, we define $ \hat{r}=\min\lbrace r+1,m\rbrace-\frac{s}{2} $. Then{ ,} there is a constant $ c>0 $ depending on $ s, ~m ,~ r $ and $ \mathcal{T}_h $ 
			such that
			\begin{equation}\label{intererror}
			\min\limits_{v_h \in S_\mathcal{T}^r(\Omega)} \Vert u-v_h \Vert_{H^{\frac{s}{2}}(\Omega)}\leq ch^{\hat{r}}\Vert u\Vert_{H^m(\Omega)},
			\end{equation}
			for all $ u\in H^{\frac{s}{2}}(\Omega)\cap H^m(\Omega).$  Particularly, for $ r=1 $ and $\phi \in H^{\frac{s}{2}}(\Omega)\cap H^\gamma(\Omega) $ where $ \gamma=\min \lbrace 2,m  \rbrace$ and $ \hat{r}=\gamma-\frac{s}{2}$, one can deduce that
			\begin{equation}\label{wd}
			\min\limits_{v_h \in U_h} \Vert \phi-v_h \Vert_{H^{\frac{s}{2}}(\Omega)}\leq ch^{\gamma-\frac{s}{2}}\Vert \phi \Vert_{H^{\gamma}(\Omega)}.
			\end{equation}
			Moreover, if $ \phi\in H^{\gamma}(\Omega)\cap V$ where $ V $ is defined in \eqref{Vv}, the following relation holds
			\begin{equation}\label{ws}
			\min\limits_{v_h \in V_h} \Vert \phi-v_h \Vert_{H^{\frac{s}{2}}(\Omega)}\leq ch^{\gamma-\frac{s}{2}}\Vert \phi\Vert_{H^\gamma(\Omega)},
			\end{equation}
		\end{lemma}
		\begin{proof}
			The relation 
			\[
			\inf_{v\in U_{h}}\Vert u-v\Vert_{H^{\alpha}(\Omega)}\leq \Vert u- \mathcal{I}_{h}u\Vert_{H^{\alpha}(\Omega)}, \quad 0\leq \alpha \leq 1,
			\]
			and the similar argument for finding an error estimation of the standard Lagrange finite element for the integer order Sobolev space { lead} to an error bound for the interpolation error in the intermediate spaces \eqref{intererror} and the special cases \eqref{wd} and \eqref{ws}; for more details see \cite{carstensen, jin2015}. 
		\end{proof}
		\begin{theorem}\label{thml}
			Assume that $ u $ { is} the exact solution of Eq. \eqref{asli} and $ u_h $ { is} the approximate solution of the variational formulation \eqref{weak} or \eqref{weak2}. Then 
			\begin{equation}
			\Vert u-u_h\Vert_{H^{\frac{s}{2}}(\Omega)}\leq Ch^{\gamma-\frac{s}{2}} \Vert u\Vert_{H^\gamma(\Omega)},
			\end{equation}
			where $ \gamma $ differs for { the} nonlinear boundary value problems with Caputo or Riemann-Liouville fractional derivative. 
			{For  the  case of Caputo differential operator, $ \gamma $ is equal to $ s $.  In addition, $ \gamma $ belongs to the interval $[\frac{s}{2}, s-\frac{1}{2}]$  for the Riemann-Liouville fractional counterpart.}
		\end{theorem}
		\begin{proof}
			Consider $ u_h\in U_h $ { is} the solution of finite element space of Eq. \eqref{asli} which satisfies  the following formulation
			\[ \mathcal{A}(u_h,v_h)+\mathcal{B}(u_h,v_h)=\langle f,v_h\rangle, \quad \quad v_h\in V_h.\]
			Next by subtracting the above equation and Eq. \eqref{weak}, we get that
			\begin{equation}\label{Au}
			\mathcal{A}(u,v)-\mathcal{A}(u_h,v_h)+\mathcal{B}(u,v)-\mathcal{B}(u_h,v_h)=\langle f,v\rangle-\langle f,v_h\rangle.
			\end{equation}
			Consider the projection operator $ \mathcal{P}_h:H^{\frac{s}{2}}(\Omega)\rightarrow U_h $ defined by 
			\begin{equation}\label{orr}
			\mathcal{A}(u,v_h)=\mathcal{A}(\mathcal{P}_hu,v_h). 
			\end{equation}
			Now by adding and subtracting $\mathcal{P}_hu$, we have 
			\begin{equation}\label{xi}u-u_h=(u-\mathcal{P}_hu)+(\mathcal{P}_hu-u_h):=\xi+\eta.
			\end{equation}
			Then,  Eq. \eqref{Au} can be rewritten as follows
			\begin{equation}
			\mathcal{A}(u,v)-\mathcal{A}(\mathcal{P}_hu,v_h)+\mathcal{A}(\mathcal{P}_hu,v_h)-\mathcal{A}(u_h,v_h)+\mathcal{B}(u,v)-\mathcal{B}(u_h,v_h)=\langle f,v\rangle-\langle f,v_h\rangle,
			\end{equation}
			therefore, from Eq. \eqref{orr} and setting $ v=v_h $, we get that
			\begin{equation}
			\mathcal{A}(\mathcal{P}_hu,v_h)-\mathcal{A}(u_h,v_h)+\mathcal{B}(u,v_h)-\mathcal{B}(u_h,v_h)=0,
			\end{equation}
			or, regarding the bilinearity of the operator $ \mathcal{A} $,
			\begin{equation}
			\mathcal{A}(\mathcal{P}_hu-u_h,v_h)+\mathcal{B}(u,v_h)-\mathcal{B}(u_h,v_h)=0.
			\end{equation}
			{ Letting} $ v_h=\eta $, we have
			\begin{equation}\label{la}
			\mathcal{A}(\eta,\eta)=\mathcal{B}(u_h,\eta)-\mathcal{B}(u,\eta).
			\end{equation}
			Next, by the coercivity of $\mathcal{A}  $, there is a constant $ c_0 $ such that
			\begin{equation}\label{coer}
			\mathcal{A}(\eta, \eta)\geq c_0\Vert \eta \Vert^2_{\tilde{H}^{\frac{s}{2}}(\Omega)}.
			\end{equation}
			On the other hand, 
			\begin{equation}\label{ahoo}
			\begin{split}
			\vert\mathcal{B}(u_h,\eta)-\mathcal{B}(u,\eta)\vert &=\vert\big(g(.,u)-g(.,u_h), \eta\big)\vert\\
			&=\vert\Big(g(.,u)-g(.,\mathcal{P}_hu)+g(.,\mathcal{P}_hu)-g(.,u_h), \eta\Big)\vert\\
			&\leq \Vert g(.,u)-g(.,\mathcal{P}_hu)\Vert \Vert \eta \Vert+\Vert g(.,\mathcal{P}_hu)-g(.,u_h)\Vert\Vert \eta \Vert\\
			&\leq l_M\Vert \eta \Vert_{H^{\frac{s}{2}}(\Omega)}(\Vert \xi \Vert_{H^{\frac{s}{2}}(\Omega)} +\Vert \eta \Vert_{H^{\frac{s}{2}}(\Omega)}),
			\end{split}
			\end{equation}
			where $ l_M $ is verified in \eqref{lip}.
			If $ \eta\neq 0 $ and $ c_{0}>l_M, $ then one can deduce from \eqref{coer} and \eqref{ahoo} that
			\begin{equation}\label{eta}
			\Vert \eta\Vert_{H^{\frac{s}{2}}(\Omega)} \leq \frac{l_M}{c_{0}-l_M}\Vert \xi\Vert_{H^{\frac{s}{2}}(\Omega)}.
			\end{equation}
			Consequently by Eqs. \eqref{xi}, \eqref{eta} and Lemma \ref{lemcars}, we get that 
			\begin{equation}
			\Vert u-u_h \Vert_{H^{\frac{s}{2}}(\Omega)} \leq C h^{\gamma-\frac{s}{2}}\Vert u\Vert_{H^\gamma(\Omega)}. \end{equation}
			It is conspicuous that for each  derivative case with different regularity, $ \gamma $ should be different. { Using the regularity Theorems \ref{thmr1} and \ref{thmr2} for $f \in L^2(\Omega)$, one can deduce that  $ \gamma \in [\frac{s}{2}, s-\frac{1}{2}]$ for { the} Riemann-Liouville  fractional derivative operator, and  
				$ \gamma=s$ in the Caputo fractional one.}
		\end{proof}
		\begin{remark}\label{reml}
			Let $f\in H^{\alpha}(\Omega)$ where $\alpha$ varies in the interval $\left[0,\frac{1}{2}\right)$ and $\alpha +s>\frac{3}{2}$. 
			If $\alpha +s<2$, then Theorem \ref{thmr2} yields that { the $\gamma$   arising in the error estimation given  for the Caputo fractional operator is equal to $ \alpha +s$ which means} that $\Vert u-u_h\Vert_{H^{\frac{s}{2}}(\Omega)}=\mathcal{O}(h^{\alpha+\frac{s}{2}})$. Otherwise, $\gamma=2$ and $\Vert u-u_h\Vert_{H^{\frac{s}{2}}(\Omega)}=\mathcal{O}(h^{2-\frac{s}{2}})$.
		\end{remark}
		{ We finalize this section with a remark on the stability of the Galerkin finite element method. 
			\begin{remark}
				By means of the triangular inequality, it is easily seen that 
				\begin{equation*}
				\Vert u_h\Vert_{H^{\frac{s}{2}}(\Omega)}\leq \Vert u-u_h\Vert_{H^{\frac{s}{2}}(\Omega)}+\Vert u\Vert_{H^{\frac{s}{2}}(\Omega)}.
				\end{equation*}
				Using Theorem \ref{thm1}, we obtain that 
				\begin{equation}
				\Vert u_h\Vert_{H^{\frac{s}{2}}(\Omega)}\leq Ch^{\gamma-\frac{s}{2}}\Vert u\Vert_{H^{\gamma}(\Omega)}+\Vert u\Vert_{H^{\frac{s}{2}}(\Omega)},
				\end{equation}
				where $C$ is a constant and $ \gamma $ is equal to $ s $, if one considers the Caputo differential operator and  $ \gamma $ belongs to $[\frac{s}{2}, s-\frac{1}{2}]$  for the Riemann-Liouville fractional operator case. Therefore, we get an upper bound for the approximate solution in the case of Caputo fractional derivative
				\begin{equation}
				\Vert u_h\Vert_{H^{\frac{s}{2}}(\Omega)}\leq Ch^{\frac{s}{2}}\Vert u\Vert_{H^{s}(\Omega)}+\Vert u\Vert_{H^{\frac{s}{2}}(\Omega)},
				\end{equation}
				and the following one in the case of Riemann-Liouville fractional derivative operator
				\begin{equation}
				\Vert u_h\Vert_{H^{\frac{s}{2}}(\Omega)}\leq C\Vert u\Vert_{H^{\frac{s}{2}}(\Omega)}.
				\end{equation}
				So the final results on the stability discussion is obtained by the argument given in Remark \ref{stab}. 
		\end{remark}} 
\section{Numerical Illustrations}\label{num}
		The numerical experiments are employed to exhibit the applicability of the Galerkin finite element method for { the} fractional nonlinear boundary value problems with Caputo and Riemann-Liouville derivatives. 
		The experiments are implemented in $\textsl{Mathematica}^{\circledR}$ software platform. We report the absolute error along with the numerical and theoretical { rates} of convergence for some examples which satisfy the assumptions considered in the previous sections.  Furthermore,  the numerical algorithm  is examined for some examples with the absence of the mentioned assumptions.
		
		In general, for the numerical experiment with { the} Galerkin method, one of the main issues is the approximation of the integrals. In our examination, the Galerkin finite element solution is obtained from the fully discrete weak from:
		
		Find $u_{h} \in A_{h}$ such that 
		\begin{equation}\label{dfemd}
		\mathcal{L}_{h}(u_{h},v_{h})=F(v_{h}), \quad \forall v_{h}\in V_{h}.
		\end{equation}
		In the above nonlinear system of equations, we have utilized Gauss-Kronrod quadrature formula to compute the integrals that need to be evaluated numerically. This { happens} mainly for the integrals involving the nonlinear term.  
		Furthermore, to solve the nonlinear system, we employ { the} Newton's iteration method. { In order to do this}, consider the bilinear form $N_{h}(u_{h};.,.)$ defined on $S^r_{\mathcal{T}}(\Omega) \times S^r_{\mathcal{T}}(\Omega)$ by
		\[
		N_{h}(u_{h};w_{h},v_{h})=(_{0}^{R}D^\frac{s}{2}_x w_{h},~ _{x}^{R}D^\frac{s}{2}_1 v_{h})+(g_{u}(.,u_{h})w_{h},v_{h}).
		\]
		The Newton's method for approximating $u_{h}$ by a sequence $\{u^k_{h}\}_{k \in \mathbb{N}}$ in $S^r_{\mathcal{T}}(\Omega)$ can be written as 
		\[
		N_{h}(u^k_{h};u^{k+1}_{h}-u^k_{h},v_{h})=F(v_{h})-\mathcal{L}_{h}(u^{k}_{h},v_{h}), \quad \forall v_{h}\in S^r_{\mathcal{T}}(\Omega),
		\]
		where $u^{0}_{h}\in S^r_{\mathcal{T}}(\Omega)$ is an initial guess  chosen by the steepest gradient algorithm.   
		
		{ Through} tables and figures, we notify the  experimental and { the} possible { theoretical} convergence rates (the reported numbers in the parentheses) for the finite element approximation of the nonlinear boundary value problems with the Riemann-Liouville and Caputo fractional derivatives in $L^2$ and $H^{\frac{s}{2}}$-norms. 
		
		\begin{example}\label{u3}
			Consider the fractional derivative equation \eqref{asli} with 
			$ g(x,u(x))=3xu^3(x)$.  The right hand side function $ f(x)$ is chosen such that 
			\begin{description}
				\item[(a)] the exact solution  with  Riemann-Liouville fractional derivative is $u(x)= \frac{1}{\Gamma(s+1)}(x^{s-1}-x^{s})$; 
				\item[(b)] the exact solution is $u(x)=\frac{1}{\Gamma(s+1)}(x-x^{s})$, where the derivative operator is of  Caputo type.
			\end{description}
			This problem satisfies the assumptions introduced in Section \ref{secconv}. Therefore, the convergence rates of { the} Caputo and { the} Riemann derivative cases in term of $ H^{\frac{s}{2}} $ error norm are $ O(h^{\frac{s}{2}}) $ and $ O(h^{\frac{s-1}{2}})$, respectively.  This  argument directly follows from  Theorem \ref{thml} for the function $ f(x) $ { that} belongs to $ L^2(\Omega).$ 
			Tables \ref{tabbc1} and \ref{tabbr1} report the $ L^2$ and $ H^{\frac{s}{2}} $ error norms for different $ s\in (1,2)$. Moreover,  we have provided some figures to exhibit both { the} theoretical and the practical rates of convergence. 	Figures \ref{fig1} and \ref{fig2} display the absolute errors for {the} Caputo and { the} Riemann-Liouville fractional derivatives with the above nonlinear term. In figures, the dashed lines show the theoretical convergence rate. 
			
		\end{example}
		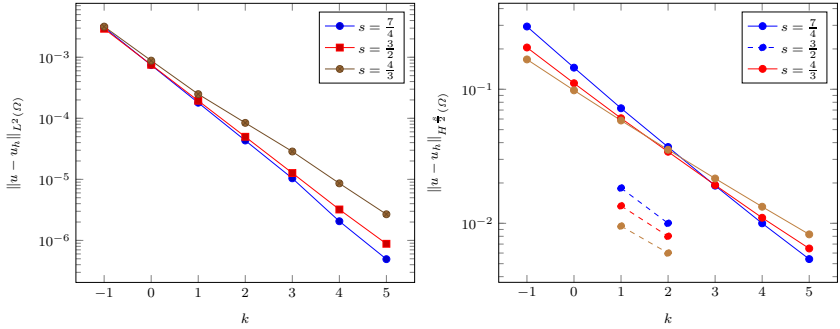
\begin{figure}[ht]
			\centering
			\begin{tikzpicture}
			[scale=0.65, transform shape]
			\begin{axis}[
			xlabel={$ k $},ylabel={$\Vert u-u_{h} \Vert_{L^{{2}}(\Omega)}$}, ymode=log, legend entries={$s=\frac{7}{4}$,$s=\frac{3}{2}$,$s=\frac{4}{3}$}, legend pos=north east]
			\addplot coordinates {
				(-1,3.08E-03)
				(0,7.44E-04)
				(1,1.80E-04)
				(2,4.33E-05)
				(3,1.04E-05)
				(4,2.06E-06)
				(5,4.93E-07)
			};
			\addplot coordinates {
				(-1,2.93E-03)
				(0,7.53E-04)
				(1, 1.94E-04)
				(2,4.98E-05)
				(3,1.27E-05)
				(4,3.21E-06)
				(5,8.80E-07)
			};
			\addplot coordinates {
				(-1,3.18E-03)
				(0,8.81E-04)
				(1, 2.49E-04)
				(2,8.42E-05)
				(3,2.86E-05)
				(4,8.60E-06)
				(5,2.68E-06)
			};
			\end{axis}
			\end{tikzpicture}
			\begin{tikzpicture}
			[scale=0.65, transform shape]
			\begin{axis}[xlabel={$ k $},ylabel={$\Vert u-u_{h} \Vert_{H^{\frac{s}{2}}(\Omega)}$}, ymode=log, legend entries={$s=\frac{7}{4}$,$s=\frac{3}{2}$,$s=\frac{4}{3}$}, legend pos=north east]
			\addplot[color=blue,mark=*] coordinates {
				(-1,2.94E-01)
				(0,1.45E-01)
				(1,7.22E-02)
				(2,3.71E-02)
				(3,1.91E-02)
				(4,1.00E-02)
				(5,5.41E-03)
			};
			\addplot[color=blue,mark=*, dashed] coordinates{
				(1,1.834E-02)
				(2,1.00E-02)
			};
			\addplot[color=red,mark=*] coordinates {
				(-1,2.05E-01)
				(0,1.110E-01)
				(1, 6.08E-02)
				(2,3.42E-02)
				(3,1.93E-02)
				(4,1.10E-02)
				(5,6.49E-03)
			};
			\addplot[color=red,mark=*, dashed] coordinates{
				(1,1.35E-02)
				(2,8.00E-03)
			};
			\addplot[color=brown,mark=*] coordinates {
				(-1,1.67E-01)
				(0,9.82E-02)
				(1, 5.85E-02)
				(2,3.54E-02)
				(3,2.16E-02)
				(4,1.33E-02)
				(5,8.27E-03)
			};	\addplot[color=brown,mark=*, dashed] coordinates{
				(1,9.54E-03)
				(2,6.00E-03)
			};
			\end{axis}
			\end{tikzpicture}
			\caption{Plots of the absolute error in  $L^{2}$  and $H^\frac{s}{2}$-norms in logarithmic scale for Example \ref{u3} with { the} Caputo fractional derivative. }%
			\label{fig1}%
		\end{figure}
		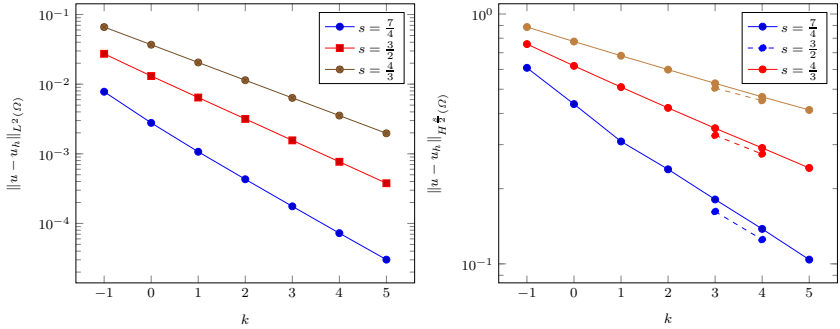
\begin{figure}[h!]
			\centering
			\begin{tikzpicture}
			[scale=0.65, transform shape]
			\begin{axis}[
			xlabel={$ k $},ylabel={$\Vert u-u_{h} \Vert_{L^{2}(\Omega)}$}, ymode=log, legend entries={$s=\frac{7}{4}$,$s=\frac{3}{2}$,$s=\frac{4}{3}$}, legend pos=north east]
			\addplot coordinates {
				(-1,7.79E-03)
				(0,2.78E-03)
				(1,1.07E-03)
				(2,4.31E-04)
				(3,1.760E-04)
				(4,7.25E-05)
				(5,3.02E-05)
			};
			\addplot coordinates {
				(-1,2.73E-02)
				(0,1.31E-02)
				(1, 6.43E-03)
				(2,3.17E-03)
				(3,1.56E-03)
				(4,7.68E-04)
				(5,3.78E-04)
			};
			\addplot coordinates {
				(-1,6.61E-02)
				(0,3.69E-02)
				(1, 2.05E-02)
				(2,1.14E-02)
				(3,6.35E-03)
				(4,3.54E-03)
				(5,1.97E-03)
			};
			\end{axis}
			\end{tikzpicture}
			\begin{tikzpicture}
			[scale=0.65, transform shape]
			\begin{axis}[
			xlabel={$ k $},ylabel={$\Vert u-u_{h} \Vert_{H^{\frac{s}{2}}(\Omega)}$}, ymode=log, legend entries={$s=\frac{7}{4}$,$s=\frac{3}{2}$,$s=\frac{4}{3}$}, legend pos=north east]
			\addplot coordinates {
				(-1,6.09E-01)
				(0,4.36E-01)
				(1,3.09E-01)
				(2,2.39E-01)
				(3,1.81E-01)
				(4,1.38E-01)
				(5,1.04E-01)
			};
			\addplot[color=blue,mark=*, dashed] coordinates{
				(3,1.62E-01)
				(4,1.25E-01)
			};
			\addplot[color=red,mark=*] coordinates {
				(-1,7.58E-01)
				(0,6.20E-01)
				(1, 5.10E-01)
				(2,4.21E-01)
				(3,3.49E-01)
				(4,2.91E-01)
				(5,2.42E-01)
			};
			\addplot[color=red,mark=*, dashed] coordinates{
				(3,3.27E-01)
				(4,2.750E-01)
			};
			\addplot[color=brown,mark=*] coordinates {
				(-1,8.87E-01)
				(0,7.76E-01)
				(1, 6.81E-01)
				(2,5.99E-01)
				(3,5.28E-01)
				(4,4.66E-01)
				(5,4.13E-01)
			};
			\addplot[color=brown,mark=*, dashed] coordinates{
				(3,5.05E-01)
				(4,4.50E-01)
			};
			\end{axis}
			\end{tikzpicture}
			\caption{Plots of the absolute error in $L^{2}$ and  $H^\frac{s}{2}$-norms  in logarithmic scale for Example \ref{u3} with { the} Riemann-Liouville fractional derivative. }%
			\label{fig2}%
		\end{figure}
		\begin{table}[!ht]
			\begin{center}
				\begin{footnotesize}
					\caption{The absolute error in $ L^2$ and $ H^{\frac{s}{2}}$-norms for different { value} $ s=\frac{7}{4},\frac{3}{2},\frac{4}{3}$ and  mesh size $h=\frac{1}{2^k\times 10} $ for Example \ref{u3} with the Caputo fractional derivative operator.}\label{tabbc1}\vspace*{0.1in}
					\begin{tabular}{|c|c|c|c|c|c|c|c|c|c|}
						\hline
						$s$& $ k$& $ -1$ & $0 $&$ 1$& $ 2$ & $3 $&$ 4$&$ 5$  & Rate\\
						\hline
						$\dfrac{7}{4}$ & $ L^2$-norm & $ 3.08\textrm{e}-03$ & $7.44\textrm{e}-04$&$ 1.80\textrm{e}-04$& $ 4.33\textrm{e}-05$ & $1.04\textrm{e}-05$&$2.06\textrm{e}-06$&$ 4.93\textrm{e}-07$ & 2.047\\
						$ $& $ H^{\frac{s}{2}}$-norm& $2.94 \textrm{e}-01$ & $1.45\textrm{e}-01 $&$ 7.22\textrm{e}-02$& $ 3.71\textrm{e}-02$ & $1.91\textrm{e}-02 $&$1.00 \textrm{e}-02$&$ 5.41\textrm{e}-03$  & 0.887 (0.875)\\
						\hline
						$\dfrac{3}{2}$& $ L^2$-norm& $ 2.93\textrm{e}-03$ & $7.53\textrm{e}-04  $&$ 1.94\textrm{e}-04$& $ 4.98\textrm{e}-05$ & $1.27\textrm{e}-05 $& $3.21\textrm{e}-06$ & $ 8.08\textrm{e}-07$ & 1.956\\
						$ $& $ H^{\frac{s}{2}}$-norm & $2.05\textrm{e}-01$ & $1.11\textrm{e}-01$&$ 6.08\textrm{e}-02$& $ 3.42\textrm{e}-02$ & $1.93\textrm{e}-02 $&$1.10\textrm{e}-02$&$ 6.49\textrm{e}-03$ & 0.769 (0.75)\\
						\hline
						$\dfrac{4}{3}$& $ L^2$-norm& $ 3.18\textrm{e}-03$ & $8.81\textrm{e}-04$&$ 2.49\textrm{e}-04$& $ 8.42\textrm{e}-05$ & $2.68\textrm{e}-05$&$ 8.60\textrm{e}-06$&$ 2.68\textrm{e}-06$  & 1.567\\
						$ $& $ H^{\frac{s}{2}}$-norm & $1.67\textrm{e}-01$ & $9.82\textrm{e}-02 $&$ 5.85\textrm{e}-02$& $3.54\textrm{e}-02$ & $2.16\textrm{e}-02 $&$1.33\textrm{e}-02$&$ 8.27\textrm{e}-03$  & 0.686 (0.67)\\\hline
					\end{tabular}
				\end{footnotesize}
			\end{center}
		\end{table}
		\begin{table}[!ht]
			\begin{center}
				\begin{footnotesize}
					\caption{The absolute error in $ L^2$ and $ H^{\frac{s}{2}}$-norms for different { values} $ s=\frac{7}{4},\frac{3}{2},\frac{4}{3}$ and  mesh size $h=\frac{1}{2^k\times 10} $ for Example \ref{u3} with the Riemann-Liouville fractional operator.}\label{tabbr1}\vspace*{0.1in}
					\begin{tabular}{|c|c|c|c|c|c|c|c|c|c|}
						\hline
						$s$& $ k$& $ -1$ & $0 $&$ 1$& $ 2$ & $3 $&$ 4$&$ 5$  & Rate\\
						\hline
						$\dfrac{7}{4}$ & $ L^2$-norm & $ 7.79\textrm{e}-03$ & $2.78\textrm{e}-03$&$ 1.07\textrm{e}-03$& $ 4.31\textrm{e}-04$ & $1.76\textrm{e}-04$&$ 7.25\textrm{e}-05$&$ 3.02\textrm{e}-05$ & 1.263 \\
						$ $& $ H^{\frac{s}{2}}$-norm & $6.09\textrm{e}-01$ & $4.36\textrm{e}-01 $&$ 3.19\textrm{e}-01$& $2.39\textrm{e}-01$ & $1.81\textrm{e}-01 $&$1.38\textrm{e}-01$&$ 1.04\textrm{e}-01$  & 0.393 (0.375)\\
						\hline
						$\dfrac{3}{2}$& $ L^2$-norm& $ 2.73\textrm{e}-02$ & $1.31\textrm{e}-02  $&$ 6.43\textrm{e}-03$& $ 3.17\textrm{e}-03$ & $1.56\textrm{e}-03 $&$ 7.68\textrm{e}-04$&$ 3.78\textrm{e}-04$ & 1.020\\
						$ $& $ H^{\frac{s}{2}}$-norm & $7.58\textrm{e}-01$ & $6.20\textrm{e}-01 $&$ 5.10\textrm{e}-01$& $4.21\textrm{e}-01$ & $3.49\textrm{e}-01 $&$2.91\textrm{e}-01$&$ 2.42\textrm{e}-01$  & 0.264 (0.250)\\
						\hline
						$\dfrac{4}{3}$& $ L^2$-norm& $ 6.61\textrm{e}-02$ & $3.69\textrm{e}-02$&$ 2.05\textrm{e}-02$& $ 1.14\textrm{e}-02$ & $6.35\textrm{e}-03$&$ 3.54\textrm{e}-03$&$ 1.97\textrm{e}-03$  & 0.840\\
						$ $& $ H^{\frac{s}{2}}$-norm & $8.87\textrm{e}-01$ & $7.76\textrm{e}-01 $&$ 6.81\textrm{e}-01$& $5.99\textrm{e}-01$ & $5.28\textrm{e}-01 $&$4.66\textrm{e}-01$&$ 4.13\textrm{e}-01$  & 0.173 (0.167)\\\hline
					\end{tabular}
				\end{footnotesize}
			\end{center}
		\end{table}
		
		\begin{example}\label{u5}
			In this example, we discuss the approximation of \eqref{asli} with 
			$ g(x,u(x))=\sin(x)u^5(x)$. The right hand side function $ f(x)$ { is} chosen such that 
			\begin{description}
				\item[(a)] the exact solution is $u(x)= \frac{\Gamma(1.5)}{\Gamma(s+1.5)}(x^{s-1}-x^{s+0.5})$ for the Riemann-Liouville case.
				\item[(b)] the exact solution is $u(x)= \frac{\Gamma(1.5)}{\Gamma(s+1.5)}(x-x^{s+0.5})$ when Eq. \eqref{asli} entails the Caputo fractional derivative.
			\end{description}
			By a similar reasoning as Example \ref{u3}, it is seen that the assumptions discussed in the theoretical parts hold. Therefore, we expect $ O(h^{\frac{s}{2}}) $ and $ O(h^{\frac{s-1}{2}})$ convergence rates for {the} Caputo and { the} Riemann-Liouville fractional differential operators, respectively. This claim is verified by the  numerical results reported { in}  	Tables \ref{tab2c} and \ref{tab2R} which exhibit the absolute errors in $ L^2$ and $H^\frac{s}{2}$-norms for different $ s\in (1,2)$.
		\end{example}
		\begin{table}[!ht]
			\begin{center}
				\begin{footnotesize}
					\caption{The absolute error in $ L^2$ and  $ H^{\frac{s}{2}}$-norms for different { values} $ s=\frac{7}{4},\frac{3}{2},\frac{4}{3}$ and  mesh size $h=\frac{1}{2^k\times 10} $ for Example \ref{u5} with the Caputo fractional operator.}\label{tab2c}\vspace*{0.1in}
					\begin{tabular}{|c|c|c|c|c|c|c|c|c|c|}
						\hline
						$s$& $ k$& $ -1$ & $0 $&$ 1$& $ 2$ & $3 $&$ 4$&$ 5$  & Rate\\
						\hline
						$\dfrac{7}{4}$ & $ L^2$-norm & $ 2.67\textrm{e}-03$ & $6.54\textrm{e}-04$&$ 1.60\textrm{e}-04$& $ 3.96\textrm{e}-05$ & $9.83\textrm{e}-06$&$2.44\textrm{e}-06$&$ 6.09\textrm{e}-07$ & 2.002\\
						$ $& $ H^{\frac{s}{2}}$-norm& $ 3.01\textrm{e}-01$ & $1.61\textrm{e}-01 $&$ 8.66\textrm{e}-02$& $ 4.67\textrm{e}-02$ & $2.52\textrm{e}-02 $&$1.36 \textrm{e}-02$&$ 7.37\textrm{e}-03$  &0.884 (0.875)\\
						\hline
						$\dfrac{3}{2}$& $ L^2$-norm& $ 2.64\textrm{e}-03$ & $6.44\textrm{e}-04  $& $ 1.59\textrm{e}-04$& $ 3.94\textrm{e}-05$ & $9.79\textrm{e}-06 $& $2.54\textrm{e}-06$ & $6.16\textrm{e}-07$ & 1.992\\
						$ $& $ H^{\frac{s}{2}}$-norm & $2.07\textrm{e}-01$ & $1.10\textrm{e}-01$&$ 5.91\textrm{e}-02$& $ 3.18\textrm{e}-02$ & $1.72\textrm{e}-02 $&$9.34\textrm{e}-03$&$5.10\textrm{e}-03$ & 0.872 (0.75)\\
						\hline
						$\dfrac{4}{3}$& $ L^2$-norm& $2.74 \textrm{e}-03$ & $6.51\textrm{e}-04$&$ 1.57\textrm{e}-04$& $ 3.84\textrm{e}-05$ & $9.53\textrm{e}-06$&$ 2.37\textrm{e}-06$&$ 5.91\textrm{e}-07$  & 2.003\\
						$ $& $ H^{\frac{s}{2}}$-norm & $1.70\textrm{e}-01$ & $9.20\textrm{e}-02 $&$5.00 \textrm{e}-02$& $2.74\textrm{e}-02$ & $1.52\textrm{e}-02 $&$8.67\textrm{e}-02$&$ 5.00\textrm{e}-03$  & 0.793 (0.67)\\\hline
					\end{tabular}
				\end{footnotesize}
			\end{center}
		\end{table}
		\begin{table}[!ht]
			\begin{center}
				\begin{footnotesize}
					\caption{The absolute errors in  $ L^2$ and $ H^{\frac{s}{2}}$-norms for different { values} $ s=\frac{7}{4},\frac{3}{2},\frac{4}{3}$ and  mesh size $h=\frac{1}{2^k\times 10} $ for Example \ref{u5} with the Riemann-Liouville fractional  operator.}\label{tab2R}\vspace*{0.1in}
					\begin{tabular}{|c|c|c|c|c|c|c|c|c|c|}
						\hline
						$s$& $ k$& $ -1$ & $0 $&$ 1$& $ 2$ & $3 $&$ 4$&$ 5$  & Rate\\
						\hline
						$\dfrac{7}{4}$ & $ L^2$-norm& $ 4.54\textrm{e}-03$ & $1.53\textrm{e}-03 $&$ 5.83\textrm{e}-04$& $ 2.35\textrm{e}-04$ & $9.73\textrm{e}-05$&$ 4.08\textrm{e}-05$&$ 1.74\textrm{e}-05$ & 1.229 \\
						$ $& $ H^{\frac{s}{2}}$-norm& $ 6.71\textrm{e}-01$ & $4.74\textrm{e}-01 $&$ 3.49\textrm{e}-01$& $ 2.62\textrm{e}-01$ & $1.97\textrm{e}-01 $&$ 1.48\textrm{e}-01$&$ 1.13\textrm{e}-01$  & 0.391 (0.375)\\
						\hline
						$\dfrac{3}{2}$& $ L^2$-norm& $ 1.56\textrm{e}-02$ & $7.55\textrm{e}-03  $&$ 3.73\textrm{e}-03$& $ 1.85\textrm{e}-03$ & $9.21\textrm{e}-03 $&$ 4.60\textrm{e}-04$&$ 2.28\textrm{e}-04$ & 1.002\\
						$ $& $ H^{\frac{s}{2}}$-norm & $ 7.64\textrm{e}-01$ & $6.27\textrm{e}-01$&$ 5.18\textrm{e}-01$& $ 4.30\textrm{e}-01$ & $3.58\textrm{e}-01 $&$ 2.98\textrm{e}-01$&$ 2.49\textrm{e}-01$ & 0.261 (0.250)\\
						\hline
						$\dfrac{4}{3}$& $ L^2$-norm& $ 4.03\textrm{e}-02$ & $2.23\textrm{e}-02 $&$ 1.24\textrm{e}-02$& $ 6.96\textrm{e}-03$ & $3.97\textrm{e}-03 $&$ 2.25\textrm{e}-03$&$ 1.29\textrm{e}-03$  & 0.801\\
						$ $& $ H^{\frac{s}{2}}$-norm & $8.57\textrm{e}-01$ & $7.52\textrm{e}-01 $&$ 6.63\textrm{e}-01$& $ 5.83\textrm{e}-01$ & $5.13\textrm{e}-01 $&$ 4.52\textrm{e}-01$&$ 3.98\textrm{e}-01$  & 0.182 (0.167)\\\hline
					\end{tabular}
				\end{footnotesize}
			\end{center}
		\end{table}
		\begin{example}\label{exexp}
			Consider the nonlinear Riemann-Liouville fractional differential equation \eqref{asli} with 
			$ g(x,u(x))=x\exp(u(x))$. The right hand side function $f(x)$ is chosen such that the exact solution $u(x)$ is 
			\[
			u(x)=
			\frac{1}{\Gamma(s+2)}(x^{s-1}-x^{s+1})-\frac{2}{\Gamma(s+3)}(x^{s-1}-x^{s+2}). 
			\]
			Tables \ref{tab3R} reports the absolute error in $ L^2$ and $ H^{\frac{s}{2}}$-norms for different $ s\in (1,2)$ with the above nonlinear term.
		\end{example}
		
		\begin{table}[!ht]
			\begin{center}
				\begin{footnotesize}
					\caption{The absolute error in $ L^2$ and $H^{\frac{s}{2}}$-norms for different { values} $ s=\frac{7}{4},\frac{3}{2},\frac{4}{3}$ and  mesh size $h=\frac{1}{2^k\times 10} $ for Example \ref{exexp} with the Riemann-Liouville fractional derivative.}\label{tab3R}\vspace*{0.1in}
					\begin{tabular}{|c|c|c|c|c|c|c|c|c|c|}
						\hline
						$s$& $ k$& $ -1$ & $0 $&$ 1$& $ 2$ & $3 $&$ 4$&$ 5$  & Rate\\
						\hline
						$\dfrac{7}{4}$ & $ L^2$-norm & $ 1.17\textrm{e}-03$ & $4.24\textrm{e}-04$&$ 1.58\textrm{e}-04$& $ 6.03\textrm{e}-05$ & $2.37\textrm{e}-05$&$9.42\textrm{e}-06$&$ 3.82\textrm{e}-06$ & 1.301\\
						$ $& $ H^{\frac{s}{2}}$-norm& $ 3.09\textrm{e}-01$ & $2.23\textrm{e}-01 $&$ 1.64\textrm{e}-01$& $ 1.21\textrm{e}-01$ & $9.10\textrm{e}-02 $&$ 6.92\textrm{e}-02$&$ 5.28\textrm{e}-02$  & 0.389 \\
						\hline
						$\dfrac{3}{2}$& $ L^2$-norm& $ 4.59\textrm{e}-03$ & $2.19\textrm{e}-03  $&$ 1.10\textrm{e}-03$& $ 5.46\textrm{e}-04$ & $2.73\textrm{e}-04 $& $1.38\textrm{e}-04$ & $6.95\textrm{e}-05$ & 0.987\\
						$ $& $ H^{\frac{s}{2}}$-norm & $4.00\textrm{e}-01$ & $3.27\textrm{e}-01$&$ 2.68\textrm{e}-01$& $2.20 \textrm{e}-01$ & $1.82\textrm{e}-01 $&$1.51\textrm{e}-01$&$1.26\textrm{e}-01$ & 0.266\\
						\hline
						$\dfrac{4}{3}$& $ L^2$-norm& $ 1.13\textrm{e}-02$ & $6.27\textrm{e}-03$&$ 3.24\textrm{e}-03$& $ 1.71\textrm{e}-03$ & $9.10\textrm{e}-04$&$ 5.06\textrm{e}-04$&$ 2.82\textrm{e}-05$  & 0.845\\
						$ $& $ H^{\frac{s}{2}}$-norm & $4.72\textrm{e}-01$ & $4.12\textrm{e}-01 $&$ 3.62\textrm{e}-01$& $3.19\textrm{e}-01$ & $2.81\textrm{e}-01 $&$2.48\textrm{e}-01$&$ 2.19\textrm{e}-01$  & 0.178 \\\hline
					\end{tabular}
				\end{footnotesize}
			\end{center}
		\end{table}
		\begin{example}\label{exu2}
			We present  the nonlinear Caputo fractional differential { equation} \eqref{asli} with 
			$ g(x,u(x))=(u(x)-x)^2, $ where the exact solution is $u(x)=\frac{\Gamma(\frac{3}{4})}{\Gamma(s+\frac{3}{4})}(x-x^{s-\frac{1}{4}})$.
			
			Table \ref{tab4c} reports the absolute errors in $ L^2$ and $ H^{\frac{s}{2}}$-norms for different $ s\in (1,2)$ with the above nonlinear term.
			\begin{table}[!ht]
				\begin{center}
					\begin{footnotesize}
						\caption{The absolute error in $ L^2$ and  $ H^{\frac{s}{2}}$-norms for different { values} $ s=\frac{7}{4},\frac{3}{2},\frac{4}{3}$ and  mesh size $h=\frac{1}{2^k\times 10} $ for the Caputo fractional operator for Example \ref{exu2}.}\label{tab4c}\vspace*{0.1in}
						\begin{tabular}{|c|c|c|c|c|c|c|c|c|c|}
							\hline
							$s$& $ k$& $ -1$ & $0 $&$ 1$& $ 2$ & $3 $&$ 4$&$ 5$  & Rate\\
							\hline
							$\dfrac{7}{4}$ & $ L^2$-norm & $ 3.94\textrm{e}-03$ & $1.01\textrm{e}-03$&$ 2.57\textrm{e}-04$& $ 6.56\textrm{e}-05$ & $1.68\textrm{e}-05$&$4.30\textrm{e}-06$&$ 1.10\textrm{e}-06$ & 1.963\\
							$ $& $ H^{\frac{s}{2}}$-norm& $ 3.44\textrm{e}-01$ & $1.79\textrm{e}-01 $&$ 9.32\textrm{e}-02$& $ 4.85\textrm{e}-02$ & $2.53\textrm{e}-02 $&$1.32 \textrm{e}-02$&$ 6.89\textrm{e}-03$  & 0.938 \\
							\hline
							$\dfrac{3}{2}$& $ L^2$-norm& $ 3.27\textrm{e}-03$ & $9.70\textrm{e}-04  $& $ 2.89\textrm{e}-04$& $ 8.59\textrm{e}-05$ & $2.56\textrm{e}-05 $& $7.61\textrm{e}-06$ & $2.26\textrm{e}-06$ & 1.749\\
							$ $& $ H^{\frac{s}{2}}$-norm & $2.24\textrm{e}-01$ & $1.32\textrm{e}-01$&$ 7.84\textrm{e}-02$& $ 4.48\textrm{e}-02$ & $2.67\textrm{e}-02 $&$1.59\textrm{e}-02$&$9.49\textrm{e}-03$ & 0.745 \\
							\hline
							$\dfrac{4}{3}$& $ L^2$-norm& $ 1.93\textrm{e}-03$ & $6.28\textrm{e}-04$&$ 2.10\textrm{e}-04$& $ 7.07\textrm{e}-05$ & $2.39\textrm{e}-05$&$8.10 \textrm{e}-06$&$ 2.76\textrm{e}-06$  & 1.551\\
							$ $& $ H^{\frac{s}{2}}$-norm & $1.32\textrm{e}-01$ & $8.48\textrm{e}-02 $&$ 5.48\textrm{e}-02$& $3.55\textrm{e}-02$ & $2.30\textrm{e}-02 $&$1.49\textrm{e}-02$&$ 9.70\textrm{e}-03$  & 0.620 \\\hline
						\end{tabular}
					\end{footnotesize}
				\end{center}
			\end{table}
		\end{example}
		\begin{example}\label{exl}
			As the final example, we deal with the linear Caputo fractional differential equation ($g(x,u(x))=0$) by considering $ f(x)=x^\theta$ { to belong} $ H^{\alpha}(\Omega)$ for $ \alpha\in \left[ 0,\theta+\frac{1}{2}\right)$ and $ \theta\in \{\frac{-1}{3},\frac{-1}{4},\frac{-1}{5} \}.$ The exact solution for different { values of} $\theta$ is $u(x)=c_\theta (x^{s-1}-x^{s+\theta})$ with $ c_\theta=\frac{\Gamma (\theta +1)}{\Gamma (s+\theta +1)}.$
			
			Table   \ref{tabl} displays the theoretical and numerical rates of convergence in $H^{\frac{s}{2}}$-norm  for different $ s=\frac{7}{4},\frac{3}{2},\frac{4}{3}$. 
			This  problem { is satisfied} the Remark \ref{reml}. For different { values of} $\alpha$ and $s$, the value of $\gamma$ is given by $ \min\{\alpha+s,2\}$. 
			For instance, for $ s=\frac{7}{4} $ and $ \theta=\frac{-1}{5},$ then $ \gamma=2$ and the rate of convergence is $ O(h^{2-\frac{s}{2}})=O(h^{1.125}). $
		\end{example}
		\begin{table}[!ht]
			\begin{center}
				\begin{small}
					\caption{A comparison between theoretical and numerical convergence rates in  $ H^{\frac{s}{2}}$-norm  for Example \ref{exl} with  the Caputo fractional derivative.}\label{tabl}\vspace*{0.1in}
					\begin{tabular}{|l|c|c|c|}
						\hline
						\diagbox[width=1.5cm, height=0.5cm]{{\large $ s $}}{{\large$ ~\theta $} }&$ 7/4 $ & $3/2 $ & $ 4/3 $ \\ \hline
						$ -1/3 $ & $\dfrac{}{} 1.059~(1.042) $ & $ 0.916~(0.917) $& $ ---- $ \\ \hline
						$ -1/4 $ & $\dfrac{}{} 1.103~(1.125) $ & $ 0.974~(1.000) $ & $ 0.925~(0.917) $\\ \hline
						$-1/5 $ & $ \dfrac{}{}1.124~(1.125) $ & $1.000 ~(1.050) $& $ 0.970~(0.967) $ \\ \hline
					\end{tabular} 
				\end{small}
			\end{center}
		\end{table}
		\section*{Conclusion and future studies} In this paper, we have studied the Lagrange finite element method for a class of semi-linear FDEs of { the}  Riemann-Liouville and { the} Caputo types. To this aim, a weak formulation of the problems { has} been introduced in the suitable function spaces constructed by considering the fractional Sobolev and Musielak-Orlicz { spaces} duo to the presence of the nonlinear term. In addition, the existence and uniqueness issue of the weak solution together with { its} regularity is discussed. The weak formulation is discretized by Galerkin method with piecewise linear polynomials basis functions. Finding  an error bound in $H^{\frac{s}{2}}$-norm is considered for the Riemann-Liouville and Caputo fractional  differential equations. Different examples with the varieties of the nonlinear terms have been examined and the absolute errors are reported in $L^2$ and $H^{\frac{s}{2}}$-norms.  
		
		The nature of the nonlinearity and also { the} fractional essence of the problem cause low order convergence of the method. 
		In order to improve the approach for this class of FDEs, one can {  apply} the idea of splitting method, where the solution is separated into regular and singular parts;  { this} is { possible} by utilizing the Taylor expansion of the nonlinear operator and the finite element method { accompanied by} a quasi-uniform mesh. 
		Also, as discussed in the numerical experiments section, the integrals in the obtained nonlinear system  are discretized by a suitable quadrature method. { Surveying} the effect of quadrature method in finite element approximation and a priori error estimation is an idea for { the} future studies. As reported in the numerical section, we have observed the absolute errors in $L^2$-norm which are sharper than the errors in $H^\frac{s}{2}$-norm. An interesting question for { a} further study is how to obtain an appropriate theoretical error bound in $L^2$-norm.

\begin{acknowledgements}
We gratefully thank Bangti Jin (University College London) for helpful discussion. 
\end{acknowledgements}



\bibliography{manabe} 
%
%

\end{document}